\documentclass[10pt]{amsart}

\usepackage{amsthm,amsfonts,amsmath,verbatim,mathrsfs,amssymb,verbatim,cite,fouridx}
\usepackage[usenames]{color}
\usepackage{todonotes}
\usepackage{hyperref}
\usepackage{breakurl}
\usepackage{url}

\newtheorem{theorem}{Theorem}[section]
\newtheorem{lemma}[theorem]{Lemma}
\newtheorem{proposition}[theorem]{Proposition}

\newtheorem{conjecture}[theorem]{Conjecture}
\theoremstyle{definition}

\numberwithin{equation}{section}

\DeclareMathOperator{\eup}{e}
\DeclareMathOperator{\Exp}{Exp}
\DeclareMathOperator{\Ell}{Ell}

\newcommand{\qhyp}[2]{\fourIdx{}{#1}{}{#2}\phi}
\renewcommand{\subseteq}{\subset}

\newcommand{\ip}[2]{\langle#1,#2\rangle}
\newcommand{\la}{\lambda}
\newcommand{\La}{\Lambda}

\newcommand{\IP}{\bar{P}^{\ast}}
\newcommand{\Mn}{\mathcal{M}_n}
\newcommand{\Hg}{\mathcal{H}}
\newcommand{\Hover}{\overline{H}}

\newcommand{\iup}{\textrm{i}}
\renewcommand{\leq}{\leqslant}
\renewcommand{\geq}{\geqslant}

\newcommand{\abs}[1]{\lvert#1\rvert}

\newcommand{\Rat}{\mathbb Q}

\newcommand{\Z}{\mathbb Z}

\newcommand{\Complex}{\mathbb C}
\newcommand{\qbin}[2]{\genfrac{[}{]}{0pt}{}{#1}{#2}}

\begin{document}

\title[A Nekrasov--Okounkov formula]{A Nekrasov--Okounkov formula for Macdonald
polynomials}

\author{Eric M. Rains}
\address{Department of Mathematics, 
California Institute of Technology,
Pasadena, CA 91125, USA}
\email{rains@caltech.edu}

\author{S. Ole Warnaar}
\address{School of Mathematics and Physics,
The University of Queensland, Brisbane, QLD 4072,
Australia}
\email{o.warnaar@maths.uq.edu.au}
\thanks{Work supported by the National Science Foundation 
(grant number DMS-1001645) and the Australian Research Council}


\begin{abstract}
We prove a Macdonald polynomial analogue of the
celebrated Nekrasov--Okounkov hook-length formula from the
theory of random partitions.
As an application we obtain
a proof of one of the main conjectures of Hausel and Rodriguez-Villegas
from their work on mixed Hodge polynomials 
of the moduli space of stable Higgs bundles on Riemann surfaces.
\end{abstract}

\maketitle

\section{Introduction}

In their paper \emph{Mixed Hodge polynomials of character varieties}
\cite{HRV08}, Hausel and Rodriguez-Villegas study the non-singular
affine variety
\[
\Mn:=\big\{A_1,B_1,\dots,A_g,B_g\in \mathrm{GL}(n,\Complex):~
(A_1,B_1)\cdots(A_g,B_g)=\zeta_n I\big\}/\!/\mathrm{GL}(n,\Complex),
\]
where $g$ is a nonnegative integer, $(A,B)$ is shorthand for the commutator 
$ABA^{-1}B^{-1}$, $\zeta_n$ is a primitive $n$th-root of unity, 
and $/\!/$ is a GIT quotient by the conjugation action of 
$\mathrm{GL}(n,\Complex)$.
$\Mn$, which is the twisted character variety of a closed Riemann
surface $\Sigma$ of genus $g$ with points the twisted homomorphisms from
$\pi_1(\Sigma)$ to $\mathrm{GL}(n,\Complex)$ modulo conjugation, 
has dimension $d_n=2n^2(g-1)+2$ ($g\geq 1$). 
For low values of the rank, $\Mn$ was previously
considered by Hitchin \cite{Hitchin87} ($n=2$) and Gothen \cite{Gothen94} 
($n=3$) in their work on the moduli space of stable 
Higgs bundles of rank $n$ on $\Sigma$. 
The main focus of Hausel and Rodriguez-Villegas is to extend the
computation of the two-variable mixed Hodge polynomial $H(\Mn;q,t)$
by Hitchin and Gothen to arbitrary $n$, and thus to obtain the
Poincar\'e and $E$-polynomials $P(\Mn;t)$ and $E(\Mn;q)$ corresponding
to the one-dimensional subfamilies
\[
P(\Mn;t)=H(\Mn;1,t)
\quad\text{and}\quad
E(\Mn;q)=q^{d_n} H(\Mn;q^{-1},-1).
\]
(For an arbitrary complex algebraic variety $X$ the mixed Hodge polynomial
is defined as the three-variable generating function $H(X;x,y,t)$ of
mixed Hodge numbers $h^{p,q;t}(X)$, but since the cohomology of
$\Mn$ is of type $(p,p)$ \cite[Corollary 4.1.11]{HRV08}, 
$h^{p,q;t}(\Mn)$ vanishes unless unless $p=q$ and one can define
$H(\Mn;q,t):=H(\Mn;q,q,t)$. In \cite[Corollary 2.2.4]{HRV08} it is also
shown that $H(\Mn;q,t)$ does not depend on the choice
of $\zeta_n$ so that $H(\Mn;q,t)$ is indeed well defined.)

Determining $H(\Mn;q,t)$ for arbitrary rank $n$ and genus $g$ is a
very hard problem. The breakthrough observation by 
Hausel and Rodriguez-Villegas is that, conjecturally, the mixed Hodge 
polynomial are related to Macdonald polynomials from the
theory of symmetric functions, resulting in an alternative 
means of computing $H(\Mn;q,t)$ as follows.
Let $\la=(\la_1,\la_2,\dots)$ be an integer partition 
identified as usual with its Young or Ferrers diagram.
For $s$ a square (in the diagram) of $\la$,
the arm-length and leg-length $a(s)=a_{\la}(s)$ and $l(s)=l_{\la}(s)$
are given by the number of boxes to the right, respectively, below $s$. 
That is, if $s$ has coordinates $(i,j)$ then $a(s)=\la_i-j$ and 
$l(s)=\la'_j-i$, where $\la'$ is the conjugate of $\la$.
For example, the arm-length and leg-length of the square $(3,3)$ in the 
partition $(8,7,7,6,4,3,1)$ 

\medskip
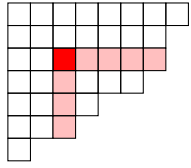
\begin{center}
\begin{tikzpicture}[scale=0.3,baseline=0cm,line width=1pt]
\draw[thin] (0,0) rectangle (1,1);
\draw[thin] (0,1) rectangle (1,2);
\draw[thin] (0,2) rectangle (1,3);
\draw[thin] (0,3) rectangle (1,4);
\draw[thin] (0,5) rectangle (1,6);
\draw[thin] (0,6) rectangle (1,7);
\draw[thin] (1,1) rectangle (2,2);
\draw[thin] (1,2) rectangle (2,3);
\draw[thin] (1,3) rectangle (2,4);
\draw[thin] (1,5) rectangle (2,6);
\draw[thin] (1,6) rectangle (2,7);
\draw[thin] (3,2) rectangle (4,3);
\draw[thin] (3,3) rectangle (4,4);
\draw[thin] (3,5) rectangle (4,6);
\draw[thin] (3,6) rectangle (4,7);
\draw[thin] (4,3) rectangle (5,4);
\draw[thin] (4,5) rectangle (5,6);
\draw[thin] (4,6) rectangle (5,7);
\draw[thin] (5,3) rectangle (6,4);
\draw[thin] (5,5) rectangle (6,6);
\draw[thin] (5,6) rectangle (6,7);
\draw[thin] (6,5) rectangle (7,6);
\draw[thin] (6,6) rectangle (7,7);
\draw[thin] (7,6) rectangle (8,7);
\fill[red](2,4) rectangle (3,5);
\begin{scope}[color=pink]
\fill(3,4) rectangle (4,5);
\fill(4,4) rectangle (5,5);
\fill(5,4) rectangle (6,5);
\fill(6,4) rectangle (7,5);
\fill(2,1) rectangle (3,2);
\fill(2,2) rectangle (3,3);
\fill(2,3) rectangle (3,4);
\end{scope}
\draw[thin](2,4) rectangle (3,5);
\draw[thin](0,4) rectangle (1,5);
\draw[thin](1,4) rectangle (2,5);
\draw[thin](3,4) rectangle (4,5);
\draw[thin](4,4) rectangle (5,5);
\draw[thin](5,4) rectangle (6,5);
\draw[thin](6,4) rectangle (7,5);
\draw[thin](2,1) rectangle (3,2);
\draw[thin](2,2) rectangle (3,3);
\draw[thin](2,3) rectangle (3,4);
\draw[thin](2,5) rectangle (3,6);
\draw[thin](2,6) rectangle (3,7);
\end{tikzpicture}
\end{center}

\noindent
are $4$ and $3$ respectively. Hausel and Rodriguez-Villegas define the
genus-$g$ hook function $\mathcal{H}_{\la}(z,w)$ as
\[
\Hg_{\la}(z,w)=
\prod_{s\in\la} \frac{(z^{2a(s)+1}-w^{2l(s)+1})^{2g}}
{(z^{2a(s)+2}-w^{2l(s)})(z^{2a(s)}-w^{2l(s)+2})},
\]
and use this to define two further families of rational functions 
$\{U_n(z,w)\}_{n\geq 1}$ and $\{\Hover_n(z,w)\}_{n\geq 1}$ by
\[
\sum_{\la} \mathcal{H}_{\la}(z,w) T^{\abs{\la}} = 
\exp\bigg(\sum_{n\geq 1} U_n(z,w)\frac{T^n}{n}\bigg),
\]
where $\abs{\la}=\la_1+\la_2+\cdots$ is the size of the partition $\la$, 
and
\begin{equation}\label{Eq_Hover}
\Hover_n(z,w):=\frac{1}{n} (z^2-1)(1-w^2) 
\sum_{d\mid n} \mu(d) U_{n/d}(z^d,w^d),
\end{equation}
with $\mu$ the M\"obius function.\footnote{Alternatively, 
$\Hover_n(z,w)$ may be defined by
\[
\sum_{\la} \mathcal{H}_{\la}(z,w) T^{\abs{\la}}= 
\Exp\bigg(\sum_{n\geq 1} \frac{\Hover_n(z,w) T^n}{(z^2-1)(1-w^2)}\bigg),
\] 
where $\Exp$ is a plethystic exponential \cite{Cadogan71,Getzler95}, 
defined for formal power series 
$f(z,w;T):=\sum_{n\geq 1} c_n(z,w) T^n$ as 
$\Exp\big(f(z,w;T)\big):=\exp\big(\sum_{r\geq 1} f(z^r,w^r;T^r)/r\big)$.}

\begin{conjecture}[{\!\!\cite[Conjecture 4.2.1]{HRV08}}]\label{Con_main}
The mixed Hodge polynomial of $\Mn$ is given by
\begin{equation}\label{Eq_HHbar}
H(\Mn;q,t)=\big(t q^{1/2}\big)^{d_n}\, \Hover_n
\big( q^{1/2},-t^{-1} q^{-1/2}\big).
\end{equation}
\end{conjecture}

This remarkable conjecture has several further implications. 
Since $H(\Mn;q,t)$ is a polynomial with positive coefficients,
\eqref{Eq_HHbar} implies that the rational function $\Hover_n(z,-w)$ 
also must be a polynomial with nonnegative coefficients. 
In the opposite direction, by $a_{\la}(s)=l_{\la'}(s)$ we have 
$\Hg_{\la}(z,w)=\Hg_{\la'}(w,z)$, which implies the ``curious 
Poincar\'e duality'' \cite[Conjecture 4.2.4]{HRV08}
\[
H(\Mn;q,t)=(qt)^{d_n} H\big(\Mn;q^{-1} t^{-2},t\big).
\]
A non-rigorous, string theoretic derivation of \eqref{Eq_HHbar} in 
the more general case of punctured Riemann surfaces \cite{HLRV11,HLRV13}
has recently been given in \cite{CDDP15}.

In the genus-$0$ case $\Mn$ has a single point for $n=1$ and no points for
higher rank. Hence $H(\Mn;q,t)=\delta_{n,1}$ and, by \eqref{Eq_Hover},
\eqref{Eq_HHbar} and $\sum_{d\mid n} \mu(d)=\delta_{n,1}$, this gives
$\Hover_n(z,w)=\Hover_{n/d}(z^d,w^d)=w^{-2n}/(1-z^{2n})(1-w^{-2n})$. For 
genus $0$ the conjecture is thus equivalent to the combinatorial
identity
\begin{equation}\label{Eq_genus0}
\sum_{\la} \Hg\big(q^{1/2},t^{-1/2}\big) T^{\abs{\la}}=
\prod_{i,j\geq 1}\frac{1}{1-q^{i-1}t^jT},
\end{equation}
which follows immediately as a special case of the Kaneko--Macdonald 
binomial theorem for Macdonald polynomials 
\cite{Kaneko96,Macdonald13}.\footnote{Hausel and Rodriguez-Villegas
prove \eqref{Eq_genus0} differently, using a duality of
Garsia and Haiman \cite{GH96} and the Cauchy identity for Schur functions.}

More interesting is the genus-$1$ case. Then $H(\Mn;q,t)=(1+qt)^2$
for all $n\geq 1$, which by \eqref{Eq_HHbar} implies 
$\Hover_n(z,w)=(z-w)^2$. Solving \eqref{Eq_Hover} for $U_n(z,w)$ leads to
\[
U_n(z,w)=\sum_{k\mid n} \frac{n}{k}\cdot 
\frac{(1-z^kw^{-k})^2}{(1-z^{2k})(1-w^{-2k})}.
\]
Since $\sum_{d\mid n}\mu(d)\sum_{m\mid (n/d)} f(md)=f(1)$, 
Conjecture~\ref{Con_main} for $g=1$ is thus
equivalent to the following combinatorial identity.

\begin{conjecture}[\!\!{\cite[Conjecture 4.3.2]{HRV08}}]\label{Con_HRV}
For $g=1$,
\[
\sum_{\la} \mathcal{H}_{\la}\big(q^{1/2},t^{-1/2}\big) T^{\abs{\la}}=
\prod_{i,j,k\geq 1} 
\frac{(1-q^{i-1/2}t^{j-1/2}T^k)^2}{(1-q^{i-1}t^{j-1}T^k)(1-q^it^jT^k)}.
\]
\end{conjecture}

In this paper we settle this conjecture 
by proving the following more general combinatorial identity.

\begin{theorem}[$q,t$-Nekrasov--Okounkov formula]\label{Thm_main}
We have
\begin{multline}\label{Eq_qtNO}
\sum_{\la} T^{\abs{\la}}
\prod_{s\in\la}\frac{(1-uq^{a(s)+1}t^{l(s)})(1-u^{-1}q^{a(s)}t^{l(s)+1})}
{(1-q^{a(s)+1}t^{l(s)})(1-q^{a(s)}t^{l(s)+1})} \\
=\prod_{i,j,k\geq 1}
\frac{(1-uq^it^{j-1}T^k)(1-u^{-1}q^{i-1}t^jT^k)}
{(1-q^{i-1}t^{j-1}T^k)(1-q^it^jT^k)}.
\end{multline}
\end{theorem}

For $u=(t/q)^{1/2}$ this is Conjecture~\ref{Con_HRV} and for general $u$ 
it is a $q,t$-analogue of the Nekrasov--Okounkov formula discovered by 
Nekrasov and Okounkov in their work on random partitions and Seiberg--Witten 
theory \cite{NO06}.
Indeed, if $h(s):=a(s)+l(s)+1$ is the hook-length of $s$ and 
$\mathscr{H}(\la):=\{h(s):s\in\la\}$ is the multiset of hook-lengths of
$\la$, then \eqref{Eq_qtNO} for $t=q$ simplifies to
\[
\sum_{\la} T^{\abs{\la}}
\prod_{h\in\mathscr{H}(\la)}\frac{(1-uq^h)(1-u^{-1}q^h)}{(1-q^h)^2}
=\prod_{k,r\geq 1} \frac{(1-uq^rT^k)^r(1-u^{-1}q^rT^k)^r}
{(1-q^{r-1}T^k)^r(1-q^{r+1}T^k)^r},
\]
first found in \cite[p.~749]{INRS12} and \cite[Theorem 5]{DH11}. 
Setting $u=q^z$ and letting $q$ tend to $1$ this yields
the Nekrasov--Okounkov formula \cite[Equation (6.12)]{NO06}
(see also \cite[Corollary 1.9]{Han10}, \cite{Westbury06})
\begin{equation}\label{Eq_NO}
\sum_{\la} T^{\abs{\la}}
\prod_{h\in\mathscr{H}(\la)}\Big(1-\frac{z^2}{h^2}\Big)=
\prod_{k\geq 1}(1-T^k)^{z^2-1}.
\end{equation}
In \cite[Proposition 6.1]{Westbury06} Westbury has shown that
for fixed $\la$ and $p$ a sufficiently large integer ($p>\abs{\la}$ suffices)
\[
\prod_{h\in\mathscr{H}(\la)}\Big(\frac{p^2}{h^2}-1\Big)
\]
is the dimension of the irreducible $\mathfrak{sl}(p,\Complex)$-module
indexed by the partition 
\[
\mu:=(\la_1,\dots,\la_p)+(\la'_1-\la'_p,\dots,\la'_1-\la'_2,0).
\]
Using the $q$-analogue of Weyl's dimension formula for 
$\mathfrak{sl}(p,\Complex)$ \cite[p.~124]{Littlewood50}
(see also \cite[Lemma 3.1]{BKW16}) 
or Stanley's hook content formula \cite[Theorem 15.3]{Stanley71}
(see also \cite[Lemma 7.21.2]{Stanley99})
it is not hard to show that in the $q$-case
\[
\prod_{h\in\mathscr{H}(\la)}\frac{(1-q^{h+p})(1-q^{h-p})}{(1-q^h)^2}=
(-1)^{\abs{\la}} q^{-l(\la)\binom{p}{2}} s_{\mu}(1,q,\dots,q^p),
\]
where $l(\la)$ is the length of $\la$ (the number of non-zero $\la_i$)
and $s_{\mu}$ a Schur function.
We did not find a similar such interpretation of the product in 
\eqref{Eq_qtNO} in terms of Macdonald polynomials.

\medskip

The remainder of this paper is organised as follows.
In the next section we first review some basic material from the theory
of Macdonald polynomials and interpolation Macdonald polynomials.
Then we apply these polynomials to prove a number of key identities needed 
in our proof of Theorem~\ref{Thm_main}. This includes the following 
elegant Cauchy-like identity for principally specialised skew Macdonald 
polynomials.
\begin{theorem}\label{Thm_Cauchy-like}
Let $\rho:=(0,1,2,\dots)$ and $q^{\rho}:=(1,q,q^2,\dots)$.
Then
\begin{multline}\label{Eq_4fold}
\sum_{\lambda,\mu,\nu,\tau}
a^{\abs{\lambda}} b^{\abs{\mu}} c^{\abs{\nu}} d^{\abs{\tau}}
b_{\nu}(q,t) b_{\tau}(t,q)
Q_{\lambda/\nu}(t^{\rho};q,t)
Q_{\lambda'/\tau}(q^{\rho};t,q) \\ \times
Q_{\mu/\nu}(t^{\rho};q;t)
Q_{\mu'/\tau}(q^{\rho};t,q) 
=\frac{1}{(abcd;abcd)_{\infty}}\cdot
\frac{(-a,-b;q,t,abcd)_{\infty}}
{(abc,abd;q,t,abcd)_{\infty}},
\end{multline}
\end{theorem}
In the above,
$b_{\la}(q,t)$ is Macdonald's $q,t$-hook function
\[
b_{\la}(q,t):=\prod_{s\in\la} \frac{1-q^{a(s)} t^{l(s)+1}}
{1-q^{a(s)+1} t^{l(s)}}
\]
and
\begin{align*}
(a;q_1,q_2,\dots,q_m)_{\infty}&:=
\prod_{i_1,\dots,i_m\geq 0}(1-a q_1^{i_1}\cdots q_m^{i_m}), \\
(a_1,\dots,a_k;q_1,q_2,\dots,q_m)_{\infty}&:=(a_1;q_1,q_2,\dots,q_m)_{\infty}
\cdots(a_k;q_1,q_2,\dots,q_m)_{\infty}
\end{align*}
are generalised $q$-shifted factorials.
In Section~\ref{Sec_fmn} we study a function $f_{n,m}$ which may 
be viewed as a rational function analogue of 
\begin{equation}\label{Eq_rat-limit}
f(u,T;q,t):=
(uq;q,t)_{\infty} 
\sum_{\la} T^{\abs{\la}}
\prod_{s\in\la}\frac{(1-uq^{a(s)+1}t^{l(s)})(1-u^{-1}q^{a(s)}t^{l(s)+1})}
{(1-q^{a(s)+1}t^{l(s)})(1-q^{a(s)}t^{l(s)+1})},
\end{equation}
see Proposition~\ref{Prop_fnm}.
We determine a number of hidden symmetries of $f_{n,m}$, 
conjecture its polynomiality, and show that up to the factor
$(uq;q,t)_{\infty}$ the limit $\lim_{n,m\to\infty} f_{n,m}$ 
is given by the product side
of \eqref{Eq_qtNO}, thus proving Theorem~\ref{Thm_main}.
Then, in Section~\ref{Sec_cases} we discuss a number of
special cases of the $q,t$-Nekrasov--Okounkov formula,
as well as a close link between our work and that of 
Iqbal, Koz\c{c}az and Shabbir \cite{IKS10}
on the topological vertex formalism.
Finally, in the appendix we give an alternative proof of the
$q,t$-Nekrasov--Okounkov formula, suggested to us by Jim Bryan, 
which is based on Waelder's equivariant Dijkgraaf--Moore--Verlinde--Verlinde 
(DMVV) formula for the Hilbert scheme of points in the plane \cite{Waelder08}.

\subsection*{Acknowledgement}
The second author is grateful to Masoud Kamgarpour for 
pointing out the papers of Hausel and Rodriguez-Villegas \cite{HRV08},
and Hausel, Letellier and Rodriguez-Villegas \cite{HLRV11,HLRV13}
on mixed Hodge polynomials, and to Dennis Stanton for helpful
discussions on $p$-core partitions.
We thank Fernando Rodriguez-Villegas for sending us a preliminary 
version of his paper \cite{CRV16} with Carlsson, which contains a different
proof of Conjecture~\ref{Con_main} based on the 
Carlsson--Nekrasov--Okounkov vertex operator \cite{CNO14}.
We also thank Amer Iqbal for alerting us to the connection 
between our work and \cite{IKS10} and Jim Bryan for explaining the
work of Waelder, which implies the elliptic Nekrasov--Okounkov formula 
described in the appendix.
We thank the two referees for their helpful comments and corrections.

\section{Macdonald polynomials}

\subsection{Partitions}

A partition $\la=(\la_1,\la_2,\dots)$ is a weakly-decreasing sequence of 
nonnegative integers such that only finitely-many $\la_i$ are non-zero.
The positive $\la_i$ are called the parts of $\la$ and the number of parts,
denoted $l(\la)$, is called the length of the partition.
If $\abs{\la}:=\la_1+\la_2+\cdots=n$ we say that $\la$ is a partition of $n$, 
and denote this by $\la\vdash n$. As is customary,
the unique partition of $0$ will be denoted by $0$.
We identify a partition $\la$ with its Young diagram
diagram consisting of $l(\la)$ left-aligned rows of squares
with $\la_i$ squares in the $i$th row.
The conjugate of $\la$, denoted $\la'$, is given by reflecting $\la$ in the 
main diagonal $i=j$, i.e., its parts are the columns of $\la$.
If $\mu$ is contained in $\la$, that is, $\mu_i\leq\la_i$ for all $i$ we write
$\mu\subseteq\la$.
We also adopt the standard dominance order on partitions, writing
$\mu\leq\la$ if and only if 
$\mu_1+\cdots+\mu_i\leq\la_1+\cdots+\la_i$ for all $i\geq 1$,
where $\la,\mu$ are partitions such that $\abs{\la}=\abs{\mu}$.
Throughout the paper we repeatedly use
\[
\delta_n:=(n-1,\dots,1,0), \qquad \rho_n:=(0,1,\dots,n-1)
\]
and 
$\rho:=(0,1,2,\dots)$. Of course, if $f(x)$ is a symmetric function, then
$f(t^{\delta_n})=f(t^{\rho_n})$.
Apart from the arm and leg lengths of a partition defined in the introduction,
we also need to arm-colength $a'(s)=a'_{\la}(s)$ 
and leg-colength $l'(s)=l'_{\la}(s)$ of $s\in\la$, given by 
the number of boxes in $\la$ immediately to the left or above $s$, 
respectively.
Equivalently, $a'(s)=j-1$ and $l'(s)=i-1$ for $s=(i,j)$.
Finally we recall the following standard statistic on partitions 
\cite{Macdonald95}:
\[
n(\la):=\sum_{s\in\la} l'(s)=
\sum_{i\geq 1} (i-1)\la_i=\sum_{i\geq 1} \binom{\la'_i}{2}.
\]

\subsection{Hook functions}

In the introduction we already defined the hook functions
$\Hg_{\la}(z,w)$ and $b_{\la}(q,t)$. We will further need
\begin{align*}
c_{\la}(q,t)&:=\prod_{s\in\la} \big(1-q^{a(s)} t^{l(s)+1}\big) \\
c'_{\la}(q,t)&:=\prod_{s\in\la} \big(1-q^{a(s)+1} t^{l(s)}\big),
\end{align*}
so that 
\[
b_{\la}(q,t)=\frac{c_{\la}(q,t)}{c'_{\la}(q,t)}
\]
and
\begin{align}\label{Eq_qfac}
(z;q,t)_{\la}&:=\prod_{s\in\la}\big(1-zq^{a'(s)}t^{-l'(s)}\big) \\
&\hphantom{:}=\prod_{i,j\geq 1} \frac{1-zq^{i-1} t^{j-\la'_i}}{1-zq^{i-1}t^j}=
\prod_{i=1}^n (zt^{1-i};q)_{\la_i}, \notag 
\end{align}
where $(z;q)_n:=(1-z)\cdots(1-zq^{n-1})$ is the usual $q$-shifted factorial.

It is easy to check from the definition that
\begin{equation}\label{Eq_cdual}
c'_{\la'}(q,t)=c_{\la}(t,q),
\end{equation}
and hence
\begin{equation}\label{Eq_bdual}
b_{\la'}(q,t)=\frac{1}{b_{\la}(t,q)}.
\end{equation}
It is also an elementary exercise to verify the relation
\begin{equation}\label{Eq_qtcdual}
c'_{\la}(1/q,1/t)=(-1)^{\abs{\la}} q^{-n(\la')-\abs{\la}}t^{-n(\la)}
c'_{\la}(q,t).
\end{equation}

\subsection{Macdonald polynomials}

Let $F=\Rat(q,t)$ and $\Lambda_{F}$ the ring of symmetric functions
in $x=(x_1,x_2,\dots)$ with coefficients in $F$. Further denote by 
$\Lambda_{n,F}$ the analogous ring over the finite alphabet
$(x_1,\dots,x_n)$.
The Newton power sums $p_{\la}$ and monomial symmetric functions $m_{\la}$
are defined as
\[
p_{\la}(x):=\prod_{i\geq 1} p_{\la_i}(x)
\]
where $p_r(x):=x_1^r+x_2^r+\cdots$ and $p_0:=1$, and
\[
m_{\la}=\sum_{\alpha} x^{\alpha},
\]
where the sum is over distinct permutations $\alpha$
of $\la$ and $x^{\alpha}:=x_1^{\alpha_1} x_2^{\alpha_2}\cdots$.
Both families of symmetric functions are bases for $\Lambda_F$.

Following Macdonald \cite{Macdonald95} we define the $q,t$-Hall scalar 
product on $\La_F$ as
\[
\ip{p_{\la}}{p_{\mu}}_{q,t}:=
\delta_{\la\mu} z_{\la} \prod_{i\geq 1}
\frac{1-q^{\la_i}}{1-t^{\la_i}},
\]
where $z_{\la}:=\prod_{i\geq 1} m_i(\la)! \, i^{m_i(\la)}$
and $m_i(\la):=\la'_i-\la'_{i+1}$.
The Macdonald polynomials $P_{\la}(q,t)=P_{\la}(x;q,t)$
are the unique family of symmetric functions such that 
\cite[p.~322]{Macdonald95}
\[
P_{\la}(q,t)=m_{\la}+\sum_{\mu<\la} u_{\la\mu}(q,t) m_{\mu},
\qquad
u_{\la\mu}\in F
\]
and
\[
\ip{P_{\la}(q,t)}{P_{\mu}(q,t)}_{q,t}=0\qquad \text{if$\quad\la\neq\mu$}.
\]
We also require the skew Macdonald polynomials 
$P_{\la/\mu}(q,t)$ defined by
\[
\ip{P_{\la/\mu}(q,t)}{P_{\nu}(q,t)}_{q,t}=
\ip{P_{\la}(q,t)}{P_{\mu}(q,t) P_{\nu}(q,t)}_{q,t}.
\]
The polynomial $P_{\la/\mu}(q,t)$ vanishes unless $\mu\subseteq\la$.
Moreover, in $\Lambda_{n,F}$, $P_{\la}(q,t)$ vanishes unless $l(\la)\leq n$.

A second family of Macdonald polynomials 
$Q_{\la/\mu}(x;q,t)=Q_{\la/\mu}(q,t)$ may be defined by
\begin{equation}\label{Eq_QP-skew}
Q_{\la/\mu}(q,t)=\frac{b_{\la}(q,t)}{b_{\mu}(q,t)}\, P_{\la/\mu}(q,t).
\end{equation}
Then $\ip{P_{\la}(q,t)}{Q_{\mu}(q,t)}_{q,t}=\delta_{\la\mu}$,
which is equivalent to the Cauchy identity \cite[p.~324]{Macdonald95}
\begin{equation}\label{Eq_Cauchy-noskew}
\sum_{\la} P_{\la}(x;q,t)Q_{\la}(y;q,t)
=\prod_{i,j\geq 1} \frac{(tx_iy_j;q)_{\infty}}{(x_iy_j;q)_{\infty}}.
\end{equation}

For Macdonald polynomials in $n$ variables we need the principal
specialisation formula \cite[p.~337]{Macdonald95}
\begin{equation}\label{Eq_MacPS}
P_{\la}(t^{\delta_n};q,t)=
t^{n(\la)} \prod_{s\in\la}
\frac{1-q^{a'(s)}t^{n-l'(s)}}{1-q^{a(s)}t^{l(s)+1}}=
t^{n(\la)} \frac{(t^n;q,t)_{\la}}{c_{\la}(q,t)}
\end{equation}
and the Macdonald--Koornwinder duality \cite[p.~332]{Macdonald95}
\begin{equation}\label{Eq_Tom}
P_{\la}(t^{\delta_n};q,t) P_{\mu}(q^{\la}t^{\delta_n};q,t)
=P_{\mu}(t^{\delta_n};q,t) P_{\la}(q^{\mu}t^{\delta_n};q,t)
\end{equation}
for $l(\la),l(\mu)\leq n$. Here $q^{\la}t^{\delta_n}:=
(q^{\la_1}t^{n-1},q^{\la_2}t^{n-2},\dots,q^{\la_n}t^0)$.

\medskip

In our proof of \eqref{Eq_4fold} it will be convenient to adopt
plethystic or $\la$-ring notation \cite{Haglund08,Lascoux01}.
In particular, for $f\in\La_F$ we use $f([(a-b)/(1-t)])$,
defined in terms of the power sums as
\begin{equation}\label{Eq_pr-plet}
p_r\Big(\Big[\frac{a-b}{1-t}\Big]\Big):=\frac{a^r-b^r}{1-t^r}.
\end{equation}
The map $\varepsilon_{a,b,t}:\La_F\to F[a,b]$ given by
$\varepsilon_{a,b,t}(f)\mapsto f([(a-b)/(1-t)])$ is a ring homomorphism,
and in particular
\begin{subequations}\label{Eq_split}
\begin{align}
P_{\la/\nu}\Big(\Big[\frac{a-b}{1-t}\Big];q,t\Big)&=
\sum_{\mu} P_{\la/\mu}\Big(\Big[\frac{a}{1-t}\Big];q,t\Big) 
P_{\mu/\nu}\Big(\Big[\frac{-b}{1-t}\Big];q,t\Big) \\&=
\sum_{\mu} P_{\la/\mu}\Big(\Big[\frac{-b}{1-t}\Big];q,t\Big)
P_{\mu/\nu}\Big(\Big[\frac{a}{1-t}\Big];q,t\Big).
\end{align}
\end{subequations}
We also note that
\begin{subequations}
\begin{align}
f\Big(\Big[a\,\frac{1-t^n}{1-t}\Big]\Big)
&=f(at^{\rho_n})=f(at^{\delta_n})\\[2mm]
f\Big(\Big[\frac{a}{1-t}\Big]\Big)&=f(at^{\rho}).
\label{Eq_plet}
\end{align}
\end{subequations}
Let $\omega_{q,t}$ be the automorphism of $\La_F$ defined by
\[
\omega_{q,t}(p_r)=(-1)^{r-1} \frac{1-q^r}{1-t^r}\, p_r.
\]
Then \cite[p.~327]{Macdonald95}
\begin{equation}\label{Eq_omega}
\omega_{q,t}\big(P_{\la/\mu}(q,t)\big)=Q_{\la'/\mu'}(t,q).
\end{equation}
If $f\in\La_F$ is homogeneous of degree $r$
then it is readily checked using \eqref{Eq_pr-plet} and \eqref{Eq_omega} 
that
\[
\varepsilon_{a,b,t}(f)=(-1)^r \varepsilon_{b,a,q}\,\omega_{q,t}(f).
\]
Applying this with $f=P_{\la/\mu}(q,t)$ and using 
\eqref{Eq_omega} implies the duality
\begin{equation}\label{Eq_abba}
P_{\la/\mu}\Big(\Big[\frac{a-b}{1-t}\Big];q,t\Big)
=(-1)^{\abs{\la}-\abs{\mu}}  
Q_{\la'/\mu'}\Big(\Big[\frac{b-a}{1-q}\Big];t,q\Big).
\end{equation}

\subsection{Interpolation Macdonald polynomials}

In this section we work exclusively in $\La_{n,F}$, and assume that
$x=(x_1,\dots,x_n)$ and $\mu$ is a partition of length at most $n$.
Then the interpolation Macdonald polynomial (or shifted Macdonald polynomial)
$\IP_{\mu}=\IP_{\mu}(x;q,t)$
is the unique (inhomogeneous) symmetric polynomial of degree $\abs{\mu}$
in $x$ such that
\begin{equation}\label{Eq_vanishing}
\IP_{\mu}(q^{\la}t^{\delta_n};q,t)=0\quad
\text{for all $\la$ such that $\mu\not\subset\la$} \\
\end{equation}
and
\begin{equation}\label{Eq_normalisation-1}
[x^{\mu}] \IP_{\mu}(x;q,t)=1.
\end{equation}
The polynomials $\IP_{\mu}(x;q,t)$ were first introduced and studied
by Knop, Okounkov and Sahi in 
\cite{Knop97,KS96,Okounkov97,Okounkov98,Sahi96}, and the choice of
defining relations differs slightly from author to author. For example,
in \eqref{Eq_vanishing} the ``for all'' condition is sometimes 
replaced by the weaker ``for all $\la\neq\mu$ such that 
$\abs{\la}\leq\abs{\mu}$''
and the normalisation \eqref{Eq_normalisation-1} is sometimes replaced by
\begin{equation}\label{Eq_normalisation-2}
\IP_{\mu}(q^{\mu} t^{\delta_n};q,t)=(-1)^{\abs{\mu}}q^{n(\mu')}
t^{(n-1)\abs{\mu}-2n(\mu)} c'_{\mu}(q,t).
\end{equation}

Below we have collected a number of results from the theory of 
interpolation Macdonald polynomials needed in our proof of the 
$q,t$-Nekrasov--Okounkov formula.
In \cite[Theorem 1.1]{Sahi96} Sahi showed that
the top-homogeneous degree term of $\IP_{\mu}(x;q,t)$ is the
Macdonald polynomial $P_{\mu}(x;q,t)$. In other words,
\begin{equation}\label{Eq_top-hom}
\lim_{a\to\infty} a^{-\abs{\mu}} \IP_{\mu}(ax;q,t)=P_{\mu}(x;q,t).
\end{equation}
For $\mu$ a partition of length at most $n$, the interpolation Macdonald 
polynomials satisfy the stability property 
\begin{equation}\label{Eq_stable}
\IP_{\mu}(tx_1,\dots,tx_n,1;q,t)=t^{\abs{\mu}} \IP_{\mu}(x_1,\dots,x_n;q,t).
\end{equation}
Okounkov \cite{Okounkov97} used this to define the $q,t$-binomial coefficients
\begin{equation}\label{Eq_qt-binom}
\qbin{\la}{\mu}_{q,t}:=
\frac{\bar{P}^{\ast}_{\mu}(q^{\la}t^{\delta_n};q,t)}
{\bar{P}^{\ast}_{\mu}(q^{\mu}t^{\delta_n};q,t)}.
\end{equation}
Thanks to \eqref{Eq_stable} the left-hand side is independent 
of $n$ as long as we take $n\geq l(\la),l(\mu)$.
It follows from the vanishing property \eqref{Eq_vanishing} that
$\qbin{\la}{\mu}_{q,t}=0$ unless $\mu\subseteq\la$.
From a duality of $\IP_{\mu}(x;q,t)$ given in
\cite[Theorem IV]{Okounkov98} Okounkov inferred the duality
\cite[Equation (2.12)]{Okounkov97}
\begin{equation}\label{Eq_qbinom-duality}
\qbin{\la}{\mu}_{q,t}=\qbin{\la'}{\mu'}_{1/t,1/q}.
\end{equation}
Finally we need the binomial theorem \cite{Okounkov97} for 
interpolation Macdonald polynomials, given by
\begin{equation}\label{Eq_Binom}
\sum_{\nu} a^{\abs{\nu}} \qbin{\la}{\nu}_{1/q,1/t}\,
\frac{\IP_{\la}(at^{-\delta_n};q,t)}{\IP_{\nu}(at^{-\delta_n};q,t)}\,
\IP_{\nu}(x;1/q,1/t)=\IP_{\la}(ax;q,t).
\end{equation}

To conclude this section we apply the binomial theorem to prove the
following sum over the product of two skew Macdonald polynomials.

\begin{proposition}\label{Prop_QQPP}
For $\la$ and $\mu$ partitions,
\begin{multline}\label{Eq_Prop}
\sum_{\nu} q^{-n(\la')-n(\mu')-\abs{\nu}} t^{n(\la)+n(\mu)}
b_{\nu}(t,q) Q_{\la'/\nu}(q^{\rho};t,q) Q_{\mu'/\nu}(q^{\rho};t,q) \\
=P_{\mu}(t^{\rho};q,t)P_{\la}(q^{-\mu} t^{\rho};q,t).
\end{multline}
\end{proposition}

\begin{proof}
Let $\la$ and $\mu$ be partitions of length at most $n$.
If we specialise $x=q^{-\mu}t^{-\delta_n}$ in the binomial theorem 
\eqref{Eq_Binom} and use definition \eqref{Eq_qt-binom} of the
$q,t$-binomial coefficient we obtain
\begin{multline*}
\sum_{\nu} a^{\abs{\nu}} 
\qbin{\la}{\nu}_{1/q,1/t} \qbin{\mu}{\nu}_{1/q,1/t}\,
\frac{\IP_{\la}(at^{-\delta_n};q,t)}{\IP_{\nu}(at^{-\delta_n};q,t)}\,
\IP_{\nu}(q^{-\nu}t^{-\delta_n};1/q,1/t) \\
=\IP_{\la}(aq^{-\mu}t^{-\delta_n};q,t).
\end{multline*}
By the duality \eqref{Eq_qbinom-duality} we may replace 
$\qbin{\la}{\nu}_{1/q,1/t} \qbin{\mu}{\nu}_{1/q,1/t}$ by 
$\qbin{\la'}{\nu'}_{t,q} \qbin{\mu'}{\nu'}_{t,q}$.
Also replacing $a$ by $at^{n-1}$, then multiplying both sides 
by $a^{-\abs{\la}}$, and finally letting $a$ tend to infinity 
using \eqref{Eq_top-hom} results in
\begin{multline}\label{Eq_byduality}
\sum_{\nu} t^{(n-1)\abs{\nu}}
\qbin{\la'}{\nu'}_{t,q} \qbin{\mu'}{\nu'}_{t,q}\,
\frac{P_{\la}(t^{\delta_n};q,t)}{P_{\nu}(t^{\delta_n};q,t)}\,
\IP_{\nu}(q^{-\nu}t^{-\delta_n};1/q,1/t)\\
=P_{\la}(q^{-\mu}t^{\rho_n};q,t).
\end{multline}
Next we use \cite[p.~323]{Lassalle98} 
(see also \cite[Equation (3.13)]{GHT99})
\begin{equation}\label{Eq_Lassalle}
\qbin{\la}{\mu}_{q,t}=t^{n(\mu)-n(\la)} 
\frac{c'_{\la}(q,t)}{c'_{\mu}(q,t)}\, Q_{\la/\mu}(t^{\rho};q,t)
\end{equation}
as well as the principal specialisation formula \eqref{Eq_MacPS} and 
normalisation formula \eqref{Eq_normalisation-2}.
This allows \eqref{Eq_byduality} to be rewritten as
\begin{multline*}
\sum_{\nu} (-1)^{\abs{\nu}}q^{n(\nu')-n(\la')-n(\mu')}
t^{n(\nu)+n(\la)} \,
\frac{(t^n;q,t)_{\la}}{(t^n;q,t)_{\nu}} \\
\times
\frac{c_{\nu}(q,t)c'_{\nu}(1/q,1/t)c'_{\la'}(t,q)c'_{\mu'}(t,q)}
{c_{\la}(q,t)(c'_{\nu'}(t,q))^2}\, 
Q_{\la'/\nu'}(q^{\rho};t,q) Q_{\mu'/\nu'}(q^{\rho};t,q) \\
=P_{\la}(q^{-\mu}t^{\rho_n};q,t).
\end{multline*}
Simplifying this using \eqref{Eq_cdual}--\eqref{Eq_qtcdual} yields
\begin{multline*}
\sum_{\nu} q^{-n(\la')-n(\mu')-\abs{\nu}} t^{n(\la)} \,
\frac{(t^n;q,t)_{\la}}{(t^n;q,t)_{\nu}} \\
\times
\frac{c_{\mu}(q,t)}{b_{\nu}(q,t)}\, 
Q_{\la'/\nu'}(q^{\rho};t,q) Q_{\mu'/\nu'}(q^{\rho};t,q)
=P_{\la}(q^{-\mu}t^{\rho_n};q,t).
\end{multline*}
Multiplying both sides by $P_{\mu}(t^{\rho_n};q,t)=P_{\mu}(t^{\delta_n};q,t)$
and once again using the principal specialisation formula \eqref{Eq_MacPS},
we finally arrive at
\begin{multline}\label{Eq_prop-finite}
\sum_{\nu} q^{-n(\la')-n(\mu')-\abs{\nu}} t^{n(\la)+n(\mu)} \,
\frac{(t^n;q,t)_{\la}(t^n;q,t)_{\mu}}{(t^n;q,t)_{\nu} \, b_{\nu}(q,t)} \\
\times 
Q_{\la'/\nu'}(q^{\rho};t,q) Q_{\mu'/\nu'}(q^{\rho};t,q)
=P_{\mu}(t^{\rho_n};q,t) P_{\la}(q^{-\mu}t^{\rho_n};q,t).
\end{multline} 
The identity \eqref{Eq_Prop} follows in the large-$n$ limit, up to the
variable change $\nu\mapsto\nu'$ and the use of 
$b_{\nu'}(q,t)b_{\nu}(t,q)=1$, see \eqref{Eq_bdual}.
\end{proof}

\subsection{Proof of Theorem~\ref{Thm_Cauchy-like}}
In this section we establish the Cauchy-like identity \eqref{Eq_4fold}
which will be key in our subsequent proof of Theorem~\ref{Thm_main}.
In fact, we will prove a slightly less-symmetric but equivalent form
obtained by the simultaneous substitution
\[
(a,b,c,d)\mapsto (Tab,cd,1/ac,1/bd).
\]

\begin{theorem}\label{Thm_4fold}
We have
\begin{multline}\label{Eq_4fold-b}
\sum_{\la,\mu,\nu,\tau} T^{\abs{\la}} b_{\nu}(q,t) b_{\tau}(t,q) 
Q_{\la/\nu}(at^{\rho};q,t) Q_{\la'/\tau}(bq^{\rho};t,q) \\ \times
Q_{\mu/\nu}(ct^{\rho};q,t) Q_{\mu'/\tau}(dq^{\rho};t,q) 
=\frac{1}{(T;T)_{\infty}}\cdot \frac{(-abT,-cd;q,t,T)_{\infty}}
{(acT,bdT;q,t,T)_{\infty}}.
\end{multline}
\end{theorem}

Before we prove this we need the following $q,t$-analogue of a
Schur function identity from page~94 of \cite{Macdonald95}.

\begin{proposition}\label{Prop_doublesum}
We have
\begin{equation}\label{Eq_doublesum}
\sum_{\la,\nu} T^{\abs{\la}} P_{\la/\nu}(x;q,t)Q_{\la/\nu}(y;q,t)
=\frac{1}{(T;T)_{\infty}}
\prod_{i,j\geq 1} 
\frac{(tTx_iy_j;q,T)_{\infty}}{(Tx_iy_j;q,T)_{\infty}}.
\end{equation}
\end{proposition}

\begin{proof}
Denote the left-hand side of \eqref{Eq_doublesum} by $f(x,y)$
and recall the generalisation of the Cauchy identity
\eqref{Eq_Cauchy-noskew} to skew functions \cite[p.~352]{Macdonald95}
\begin{multline}\label{Eq_Cauchy}
\sum_{\la} P_{\la/\nu}(x;q,t)Q_{\la/\tau}(y;q,t) \\
=\prod_{i,j\geq 1} \frac{(tx_iy_j;q)_{\infty}}{(x_iy_j;q)_{\infty}}\,
\sum_{\la} P_{\tau/\la}(x;q,t)Q_{\nu/\la}(y;q,t).
\end{multline}
Applying this with $(x,\tau)\mapsto (Tx,\nu)$ 
and multiplying both sides by $T^{\abs{\nu}}$, 
it follows from the homogeneity of the Macdonald polynomials
that
\[
f(x,y)=\prod_{i,j\geq 1} 
\frac{(tTx_iy_j;q)_{\infty}}{(Tx_iy_j;q)_{\infty}}\,
\sum_{\la,\nu}
T^{\abs{\nu}} P_{\nu/\la}(Tx;q,t) Q_{\nu/\la}(y;q,t).
\]
By the simultaneous variable change $(\la,\nu)\mapsto(\nu,\la)$
this yields 
\[
f(x,y)=f(Tx,y) \prod_{i,j\geq 1} 
\frac{(tTx_iy_j;q)_{\infty}}{(Tx_iy_j;q)_{\infty}}
\]
and thus
\[
f(x,y)=f(0,y) \prod_{i,j\geq 1} 
\frac{(tTx_iy_j;q,T)_{\infty}}{(Tx_iy_j;q,T)_{\infty}}.
\]
By $P_{\la/\nu}(0;q,t)=\delta_{\la\nu}$ and $Q_{\la/\la}(y;q,t)=1$
we finally get
\[
f(0,y)=\sum_{\la} T^{\abs{\la}}=\frac{1}{(T;T)_{\infty}},
\]
and the claim follows.
\end{proof}

\begin{proof}[Proof of Theorem~\ref{Thm_4fold}]
If we take Proposition~\ref{Prop_doublesum}, replace $\nu\mapsto\mu$,
and then make the plethystic substitutions $x\mapsto (a-d)/(1-t)$ and
$y\mapsto (c-b)/(1-t)$, we get
\[
\sum_{\la,\mu} T^{\abs{\la}} P_{\la/\mu}\Big(\Big[\frac{a-d}{1-t}\Big];q,t\Big)
Q_{\la/\mu}\Big(\Big[\frac{c-b}{1-t}\Big];q,t\Big)
=\frac{1}{(T;T)_{\infty}}\cdot
\frac{(abT,cdT;q,t,T)_{\infty}}{(acT,bdT;q,t,T)_{\infty}}.
\]
Here the product on the right follows from \cite[p.~310]{Macdonald95}
\begin{equation}\label{Eq_prod}
\prod_{i,j\geq 1}
\frac{(tTx_iy_j;q)_{\infty}}{(Tx_iy_j;q)_{\infty}}=
\exp\bigg( \sum_{r\geq 1}\frac{T^r}{r} \cdot \frac{1-t^r}{1-q^r}\,
p_r(x) p_r(y)\bigg),
\end{equation}
equation~\eqref{Eq_pr-plet} and\footnote{This may be stated more simply
as $\Exp\big({\mp} T/(1-q)(1-t)\big)=(T;q,t)_{\infty}^{\pm}$.}
\[
\exp\bigg(\sum_{r\geq 1} \frac{\mp T^r}{r(1-q^r)(1-t^r)}\bigg)=
(T;q,t)_{\infty}^{\pm 1}.
\]
Using both equations in \eqref{Eq_split} (which also hold with $P$ replaced 
by $Q$) gives
\begin{multline*}
\sum_{\lambda,\mu,\nu,\tau}
T^{\abs{\la}} 
P_{\la/\nu}\Big(\Big[\frac{a}{1-t}\Big];q,t\Big)
Q_{\la/\tau}\Big(\Big[\frac{-b}{1-t}\Big];q,t\Big) \\ \times
P_{\nu/\mu}\Big(\Big[\frac{-d}{1-t}\Big];q,t\Big)
Q_{\tau/\mu}\Big(\Big[\frac{c}{1-t}\Big];q,t\Big)
=\frac{1}{(T;T)_{\infty}}\cdot
\frac{(abT,cdT;q,t,T)_{\infty}}{(acT,bdT;q,t,T)_{\infty}}.
\end{multline*}
Transforming the sum over $\mu$ by the Cauchy identity
\eqref{Eq_Cauchy} with $(\la,\nu,\tau)\mapsto(\mu,\tau,\nu)$,
$x\mapsto -d/(1-t)$ and $y\mapsto c/(1-t)$ leads to
\begin{multline*}
\sum_{\la,\mu,\nu,\tau}
T^{\abs{\la}}
P_{\la/\nu}\Big(\Big[\frac{a}{1-t}\Big];q,t\Big)
Q_{\la/\tau}\Big(\Big[\frac{-b}{1-t}\Big];q,t\Big) \\ \times
Q_{\mu/\nu}\Big(\Big[\frac{c}{1-t}\Big];q,t\Big)
P_{\mu/\tau}\Big(\Big[\frac{-d}{1-t}\Big];q,t\Big) 
=\frac{1}{(T;T)_{\infty}}\cdot
\frac{(abT,cd;q,t,T)_{\infty}}{(acT,bdT;q,t,T)_{\infty}},
\end{multline*}
where this time we have used \eqref{Eq_prod} with $T=1$.
By the duality \eqref{Eq_abba} this is also
\begin{multline*}
\sum_{\la,\mu,\nu,\tau}
(-1)^{\abs{\la}+\abs{\mu}}
T^{\abs{\la}}
P_{\la/\nu}\Big(\Big[\frac{a}{1-t}\Big];q,t\Big)
P_{\la'/\tau'}\Big(\Big[\frac{b}{1-q}\Big];t,q\Big) \\ \times
Q_{\mu/\nu}\Big(\Big[\frac{c}{1-t}\Big];q,t\Big)
Q_{\mu'/\tau'}\Big(\Big[\frac{d}{1-q}\Big];t,q\Big) 
=\frac{1}{(T;T)_{\infty}}\cdot
\frac{(abT,cd;q,t,T)_{\infty}}{(acT,bdT;q,t,T)_{\infty}}.
\end{multline*}
Replacing $(b,d;\tau)\mapsto (-b,-d;\tau')$, and using
\eqref{Eq_QP-skew} and \eqref{Eq_plet} completes the proof
of \eqref{Eq_4fold-b}.
\end{proof}

\section{Proof of Theorem~\ref{Thm_main}}\label{Sec_fmn}

Instead of giving a direct proof of the $q,t$-Nekrasov--Okounkov
formula we will first study a rational function 
$f_{n,m}$ which, as follows from Proposition~\ref{Prop_fnm} below, 
may be viewed as a rational function analogue of the sum side 
of \eqref{Eq_qtNO}. As it turns out, almost all of our steps towards
proving Theorem~\ref{Thm_main} can be carried out at the level of
$f_{n,m}$. 

Let
\begin{multline}\label{Eq_fdef}
f_{n,m}(u,T;q,t):=(-u)^{nm} q^{n\binom{m+1}{2}} t^{m\binom{n}{2}} \\
\times
\sum_{\la,\mu\subseteq (m^n)} T^{\abs{\la}}
(-uq^mt^{n-1})^{-\abs{\la}-\abs{\mu}} 
P_{\la}(t^{\delta_n};q,t) P_{\la'}(q^{\delta_m};t,q) \\
\times
P_{\mu}(q^{\la} t^{\delta_n};q,t) P_{\mu'}(t^{\la'}q^{\delta_m};t,q).
\end{multline}

An obvious symmetry of $f_{n,m}$ is
\[
f_{n,m}(u,T;q,t)=f_{m,n}(uq/t,T;t,q).
\]
Not as apparent are the following two additional symmetries.
\begin{lemma}
We have
\begin{subequations}
\begin{align}\label{Eq_symm1}
f_{n,m}(u,T;q,t)&=T^{nm} f_{n,m}(u/T,1/T;q,t) \\[2mm]
&=f_{n,m}(tT/uq,T;q,t) \label{Eq_symm2}.
\end{align}
\end{subequations}
\end{lemma}

\begin{proof}
By the duality \eqref{Eq_Tom} we get
\begin{multline*}
f_{n,m}(u,T;q,t)= (-u)^{nm} q^{n\binom{m+1}{2}} t^{m\binom{n}{2}} \\
\times
\sum_{\la,\mu\subseteq (m^n)} T^{\abs{\la}}
(-uq^mt^{n-1})^{-\abs{\la}-\abs{\mu}} 
P_{\mu}(t^{\delta_n};q,t) P_{\mu'}(q^{\delta_m};t,q) \\
\times
P_{\la}(q^{\mu} t^{\delta_n};q,t) P_{\la'}(t^{\mu'}q^{\delta_m};t,q).
\end{multline*}
Renaming the summation index $\la$ as $\mu$ and vice versa yields
\begin{multline*}
f_{n,m}(u,T;q,t)= (-u)^{nm} q^{n\binom{m+1}{2}} t^{m\binom{n}{2}} \\
\times
\sum_{\la,\mu\subseteq (m^n)} T^{\abs{\mu}}
(-uq^mt^{n-1})^{-\abs{\la}-\abs{\mu}} 
P_{\la}(t^{\delta_n};q,t) P_{\la'}(q^{\delta_m};t,q) \\
\times
P_{\mu}(q^{\la} t^{\delta_n};q,t) P_{\mu'}(t^{\la'}q^{\delta_m};t,q).
\end{multline*}
Comparing this with \eqref{Eq_fdef} implies \eqref{Eq_symm1}.

\medskip

Next we replace the sum over $\mu$ in \eqref{Eq_fdef} by a sum over 
its complement with respect to $(m^n)$, denoted by 
$\tilde{\mu}=(m-\mu_n,\dots,m-\mu_1)$.
Recalling that (see e.g., \cite[Equation (4.3)]{BF99})
\[
P_{\tilde{\mu}}(x_1,\dots,x_n;q,t)=(x_1\cdots x_n)^m 
P_{\mu}(x_1^{-1},\dots,x_n^{-1};q,t),
\]
this yields
\begin{multline*}
f_{n,m}(u,T;q,t)=
\sum_{\la,\mu\subseteq (m^n)} 
(-tT/u)^{\abs{\la}} (-uq^mt^{n-1})^{\abs{\mu}} 
P_{\la}(t^{\delta_n};q,t) P_{\la'}(q^{\delta_m};t,q) \\ \times
P_{\mu}(q^{-\la} t^{-\delta_n};q,t) P_{\mu'}(t^{-\la'}q^{-\delta_m};t,q).
\end{multline*}
Since 
$P_{\la}(t^{\delta_n};q,t)=P_{\la}(t^{\rho_n};q,t)$
and
$P_{\mu}(q^{-\la} t^{-\delta_n};q,t)=t^{(1-n)\abs{\mu}}
P_{\mu}(q^{-\la} t^{\rho_n};q,t)$,
this can be further transformed into
\begin{multline}\label{Eq_beforelimit}
f_{n,m}(u,T;q,t)=
\sum_{\la,\mu\subseteq (m^n)} (-tT/u)^{\abs{\la}} (-uq)^{\abs{\mu}} 
P_{\la}(t^{\rho_n};q,t) P_{\la'}(q^{\rho_m};t,q) \\ \times
P_{\mu}(q^{-\la} t^{\rho_n};q,t) P_{\mu'}(t^{-\la'}q^{\rho_m};t,q).
\end{multline}
Applying the symmetry $P_{\la}(x;q,t)=P_{\la}(x;1/q,1/t)$ 
(see \cite[p.~324]{Macdonald95}) to \eqref{Eq_Tom}
and then replacing $(q,t)$ by $(1/q,1/t)$ we obtain
\[
P_{\la}(t^{\rho_n};q,t) P_{\mu}(q^{-\la}t^{\rho_n};q,t)
=P_{\mu}(t^{\rho_n};q,t) P_{\la}(q^{-\mu}t^{\rho_n};q,t).
\]
Using this in \eqref{Eq_beforelimit} and then again swapping
$\la$ and $\mu$, we find
\begin{multline*}
f_{n,m}(u,T;q,t)=
\sum_{\la,\mu\subseteq (m^n)} (-tT/u)^{\abs{\mu}} (-uq)^{\abs{\la}} 
P_{\la}(t^{\rho_n};q,t) P_{\la'}(q^{\rho_m};t,q) \\ \times
P_{\mu}(q^{-\la} t^{\rho_n};q,t) P_{\mu'}(t^{-\la'}q^{\rho_m};t,q).
\end{multline*}
Comparing the above with \eqref{Eq_beforelimit} yields \eqref{Eq_symm2}.
\end{proof}

Next we compute $f_{n,m}$ in two different ways.
First, using the homogeneity of the Macdonald polynomials and the
dual Cauchy identity \cite[p.~329]{Macdonald95}
\begin{equation}\label{Eq_dual-Cauchy}
\sum_{\mu} T^{\abs{\mu}}P_{\mu}(x;q,t) P_{\mu'}(y;t,q)=
\prod_{i,j\geq 1} (1+Tx_i y_j),
\end{equation}
we can perform the sum over $\mu$ in \eqref{Eq_fdef}. Also using
\begin{multline*}
\prod_{i=1}^n \prod_{j=1}^m \big(1-u^{-1} q^{\la_i-j} t^{\la'_j-i+1}\big)
=(-u)^{-nm} q^{-n\binom{m+1}{2}}t^{-m\binom{n}{2}}
(q^mt^n)^{\abs{\la}} \\ \times
\prod_{i=1}^n \prod_{j=1}^m \big(1-uq^{j-\la_i} t^{i-\la'_j-1}\big),
\end{multline*}
this gives
\begin{multline}\label{Eq_befrem}
f_{n,m}(u,T;q,t)=
\sum_{\la\subseteq (m^n)} (-tT/u)^{\abs{\la}} 
P_{\la}(t^{\delta_n};q,t) P_{\la'}(q^{\delta_m};t,q) \\ \times
\prod_{i=1}^n \prod_{j=1}^m \big(1-u q^{j-\la_i} t^{i-\la'_j-1}\big).
\end{multline}
Before we proceed we remark that if $u=t$ then the summand contains the factor
\[
\prod_{i=1}^n \prod_{j=1}^m \big(1-q^{j-\la_i} t^{i-\la'_j}\big).
\]
This vanishes for all $\la\subseteq (m^n)$ with the exception of $\la=0$.
Similarly, if $u=1/q$ then the summand contains the factor
\[
\prod_{i=1}^n \prod_{j=1}^m \big(1-q^{j-\la_i-1} t^{i-\la'_j-1}\big),
\]
which vanishes unless $\la=(m^n)$. 
Finally, if we replace $T$ by $uT$ and then let $u$ tend to $0$ we are 
left with
\[
\lim_{u\to 0} f_{n,m}(u,uT;q,t)=
\sum_{\la\subseteq (m^n)} 
(-tT)^{\abs{\la}} 
P_{\la}(t^{\delta_n};q,t) P_{\la'}(q^{\delta_m};t,q),
\]
which can be summed by \eqref{Eq_dual-Cauchy}.
We summarise these three observations
in the following lemma, where we have also used that 
$P_{(m^n)}(x_1,\dots,x_n;q,t)=(x_1\cdots x_n)^m$.
\begin{lemma}
We have
\begin{subequations}
\begin{align}
\label{Eq_feval1}
f_{n,m}(t,T;q,t)&=\prod_{i=1}^n \prod_{j=1}^m (1-q^jt^i) \\
f_{n,m}(1/q,T;q,t)&=T^{mn} \prod_{i=1}^n \prod_{j=1}^m (1-q^j t^i) \\
\lim_{u\to 0} f_{n,m}(u,uT;q,t)&=\prod_{i=1}^n \prod_{j=1}^m (1-Tq^{j-1} t^i).
\label{Eq_feval3}
\end{align}
\end{subequations}
\end{lemma}

Returning to \eqref{Eq_befrem}, we use
\begin{multline*}
\prod_{i=1}^n \prod_{j=1}^m \big(1-u q^{j-\la_i} t^{i-\la'_j-1}\big) =
(-u)^{\abs{\la}} q^{-n(\la')} t^{-n(\la)-\abs{\la}} \\
\times
\prod_{i=1}^n \prod_{j=1}^m \big(1-u q^j t^{i-1}\big)
\prod_{s\in\la} 
\frac{(1-uq^{a(s)+1}t^{l(s)})(1-u^{-1} q^{a(s)} t^{l(s)+1})}
{(1-uq^{m-a'(s)} t^{l'(s)})(1-u q^{a'(s)+1} t^{n-l'(s)-1})}
\end{multline*}
together with the principal specialisation formula \eqref{Eq_MacPS} 
to obtain the following.

\begin{proposition}\label{Prop_fnm}
The rational function $f_{n,m}$ can be expressed as
\begin{multline*}
f_{n,m}(u,T;q,t)=
\prod_{i=1}^n \prod_{j=1}^m \big(1-u q^j t^{i-1}\big) \\
\times
\sum_{\la\subseteq (m^n)} 
T^{\abs{\la}} 
\prod_{s\in\la}\bigg(
\frac{(1-q^{a'(s)}t^{n-l'(s)})(1-q^{m-a'(s)}t^{l'(s)})}
{(1-u q^{a'(s)+1} t^{n-l'(s)-1})(1-uq^{m-a'(s)} t^{l'(s)})} \\
\times
\frac{(1-uq^{a(s)+1}t^{l(s)})(1-u^{-1} q^{a(s)} t^{l(s)+1})}
{(1-q^{a(s)+1}t^{l(s)})(1-q^{a(s)}t^{l(s)+1})} \bigg).
\end{multline*}
\end{proposition}

As an immediate application of the proposition we have
that $f(u,T;q,t)$ defined in \eqref{Eq_rat-limit} is given by
\[
f(u,T;q,t)=\lim_{n,m\to\infty} f_{n,m}(u,T;q,t).
\]

\medskip 

For our second computation of $f_{n,m}$ we start with 
the representation given in \eqref{Eq_beforelimit}
and twice apply the finite form of Proposition~\ref{Prop_QQPP} given by
\eqref{Eq_prop-finite}. Then 
\begin{align*}
f_{n,m}(u,T;q,t)&=
\sum_{\la,\mu,\nu,\tau\subseteq (m^n)} (-tT/u)^{\abs{\la}} (-uq)^{\abs{\mu}} 
q^{-\abs{\tau}} t^{-\abs{\nu}} 
\\ & \qquad \quad\times 
\frac{(t^n;q,t)_{\la}(t^n;q,t)_{\mu}}{(t^n;q,t)_{\tau'} \, b_{\tau'}(q,t)} 
\cdot
\frac{(q^m;t,q)_{\la'}(q^m;t,q)_{\mu'}}{(q^m;t,q)_{\nu'} \, b_{\nu'}(t,q)} 
\\[1mm]
& \qquad \quad \times 
Q_{\la/\nu}(t^{\rho};q,t) Q_{\la'/\tau}(q^{\rho};t,q) 
Q_{\mu/\nu}(t^{\rho};q,t) Q_{\mu'/\tau}(q^{\rho};t,q).
\end{align*}
Since we do not know of a suitable finite analogue of 
\eqref{Eq_4fold-b}, we next let $n$ and $m$ tend to infinity,
and use \eqref{Eq_bdual} as well as the
homogeneity of the Macdonald polynomials. This yields
\begin{multline*}
f(u,T;q,t)=
\sum_{\la,\mu,\nu,\tau\subseteq (m^n)} T^{\abs{\la}} 
b_{\nu}(q,t) b_{\tau}(t,q) 
Q_{\la/\nu}(-t^{\rho+1}/u;q,t) Q_{\la'/\tau}(q^{\rho};t,q) 
\\ \times
Q_{\mu/\nu}(-ut^{\rho};q,t) Q_{\mu'/\tau}(q^{\rho+1};t,q).
\end{multline*}
By \eqref{Eq_4fold-b} with $(a,b,c,d)=(-t/u,1,-u,q)$ the sum
evaluates in closed form as
\begin{align}\label{Eq_fproduct}
f(u,T;q,t)&=\frac{1}{(T;T)_{\infty}}\cdot 
\frac{(uq,u^{-1}tT;q,t,T)_{\infty}}
{(qT,tT;q,t,T)_{\infty}} \\
&=\frac{(uq,u^{-1}tT;q,t,T)_{\infty}}{(T,qtT;q,t,T)_{\infty}}.
\notag
\end{align}
Equating this with \eqref{Eq_rat-limit} and dividing both sides by
$(uq;q,t)_{\infty}$ results in \eqref{Eq_qtNO}.

\bigskip

To conclude this section we make some final remarks about 
$f_{n,m}$ and why it is an interesting function in its own right.
First of all, from \eqref{Eq_fdef} we have 
$f_{n,m}(u,T;q,t)\in\Rat(q,t)[u,u^{-1},T]$.
A stronger result appears to hold as follows.

\begin{conjecture}
The function $f_{n,m}(u,T;q,t)$ lies in $\Z[q,t,u,u^{-1},T]$.
\end{conjecture}

For $n=1$ (or $m=1$) this is easily seen to be true. For example,
taking $n=1$ in Proposition~\ref{Prop_fnm} yields
\begin{equation}\label{Eq_fniseen}
f_{1,m}(u,T;q,t)=\sum_{k=0}^m T^k \qbin{m}{k}_q (t/u;q)_k (uq;q)_{m-k},
\end{equation}
where $\qbin{m}{k}_q$ is the classical $q$-binomial coefficient
(see e.g., \cite[Equation (I.39)]{GR04}).
Since the summand lies in $\Z[q,t,u,u^{-1},T]$,
so does $f_{1,m}(u,T;q,t)$.
For $n,m>1$, however, the conjectured polynomiality is much deeper.

The function $f_{n,m}$ may be viewed as a generalised  
basic hypergeometric series, and some of the symmetries and evaluations 
proved in this section are generalisations of well-known results 
for such series.
Defining \cite[p.~68]{Rains05}
\[
C_{\la}^{-}(z;q,t):=
\prod_{s\in\la} \big(1-zq^{a(s)} t^{l(s)}\big),
\]
we may rewrite Proposition~\ref{Prop_fnm} as 
\begin{multline*}
f_{n,m}(u,T;q,t)=
\prod_{i=1}^n \prod_{j=1}^m \big(1-u q^j t^{i-1}\big) \\
\times
\sum_{\la\subseteq (m^n)} 
\frac{C_{\la}^{-}(uq;q,t) C_{\la}^{-}(t/u;q,t)}
{C_{\la}^{-}(q;q,t) C_{\la}^{-}(t;q,t)}\cdot
\frac{(t^n;q,t)_{\la} (q^{-m};q,t)_{\la}}
{(uqt^{n-1};q,t)_{\la} (q^{-m}/u;q,t)_{\la}}\,
\Big(\frac{T}{u}\Big)^{\abs{\la}}.
\end{multline*}
For $n=1$ (see also \eqref{Eq_fniseen}) this simplifies to
the $\qhyp{2}{1}$ basic hypergeometric series
\begin{equation}\label{Eq_qhyp}
f_{1,m}(u,T;q,t)=(uq;q)_m \:
\qhyp{2}{1}\bigg[\genfrac{}{}{0pt}{}
{t/u,q^{-m}}{q^{-m}/u};q,\frac{T}{u}\bigg],
\end{equation}
where \cite{GR04}
\[
\qhyp{r+1}{r}\bigg[\genfrac{}{}{0pt}{}
{a_1,\dots,a_{r+1}}{b_1,\dots,b_r};q,z\bigg]:=
\sum_{k=0}^{\infty} \frac{(a_1;q)_k\cdots (a_{r+1};q)_k}
{(q;q)_k (b_1;q)_k\cdots (b_r;q)_k}\, z^k.
\]
The symmetry \eqref{Eq_symm1} can thus be viewed as
a generalisation of Heine's well-known transformation 
formula \cite[Equation (III.2)]{GR04}
\[
\qhyp{2}{1}\bigg[\genfrac{}{}{0pt}{}
{aw,q^{-m}}{wq^{-m}};q,z\bigg]=
\Big(\frac{z}{w}\Big)^m \frac{(q/z;q)_m}{(q/w;q)_m}\:
\qhyp{2}{1}\bigg[\genfrac{}{}{0pt}{}
{az,q^{-m}}{zq^{-m}};q,w\bigg].
\]
Similarly, \eqref{Eq_symm2} generalises
\[
\qhyp{2}{1}\bigg[\genfrac{}{}{0pt}{}
{q/w,q^{-m}}{aq^{-m}/w};q,z\bigg]=
\frac{(zq/a;q)_m}{(wq/a;q)_m}\:
\qhyp{2}{1}\bigg[\genfrac{}{}{0pt}{}
{q/z,q^{-m}}{aq^{-m}/z};q,w\bigg],
\]
which is a limiting case of the $\qhyp{3}{2}$
transformation formula \cite[Equation (III.11)]{GR04}.
Also, by \eqref{Eq_symm2}, the evaluation \eqref{Eq_feval1}
is equivalent to
\begin{equation}\label{Eq_feval2}
f_{n,m}(T/q,T;q,t)=\prod_{i=1}^n \prod_{j=1}^m (1-q^jt^i).
\end{equation}
Comparing this with the $u=T/q$ case of \eqref{Eq_qhyp}
shows that \eqref{Eq_feval2} can be viewed as generalisations 
of the $q$-Chu--Vandermonde summation \cite[Equation (II.6)]{GR04}
\[
\qhyp{2}{1}\bigg[\genfrac{}{}{0pt}{}
{a,q^{-m}}{bq^{-m}};q,q\bigg]=
\frac{(aq/b;q)_m}{(q/b;q)_m}.
\]
Finally we note that if we let $m$ tend to infinity in
\eqref{Eq_qhyp} we can sum the resulting $\qhyp{1}{0}$ series
by the  
$q$-binomial theorem \cite[Equation (II.3)]{GR04}. Hence
\begin{equation}\label{Eq_limf1m}
\lim_{m\to\infty} f_{1,m}(u,T;q,t)
=\frac{(uq,tT/u;q)_{\infty}}{(T;q)_{\infty}}.
\end{equation}
Unfortunately, $\lim_{m\to\infty} f_{n,m}(u,T;q,t)$ for finite $n>1$,
which interpolates between 
\eqref{Eq_fproduct} and \eqref{Eq_limf1m},
does not admit a simple factorised expression.

\section{Special cases of the $q,t$-Nekrasov--Okounkov formula}\label{Sec_cases}

The Nekrasov--Okounkov formula \eqref{Eq_NO} contains many classical
identities as special cases. For $\mu=0$ it yields Euler's formula
for the generating function of partitions. For $z=2$ only 
the staircase partitions $\delta_n$ for $n\geq 1$ 
contribute to the sum and \eqref{Eq_NO} simplifies to
Jacobi's identity for the third power of the Dedekind eta function 
$\eta(\tau)$.
More generally, for $z=p$ with $p$ a positive integer, \eqref{Eq_NO}
it is related to Macdonald's expansion \cite[pp.~134 and 135]{Macdonald72}
for the $(p^2-1)$th power of $\eta(\tau)$.
In a different vein (see e.g., \cite{Han10}), 
by setting $z^2=-x/T$, taking the $T\to 0$ limit,
and then extracting coefficients of $x^n$, 
the Nekrasov--Okounkov formula simplifies to
\begin{equation}\label{Eq_nfactorial}
\sum_{\la\vdash n}
\prod_{s\in\la}\frac{1}{h(s)^2}=\frac{1}{n!},
\end{equation}
which is a well-known identity related to the Robinson--Schensted--Knuth
correspondence \cite{Knuth70,Robinson38,Schensted61},
the Frame--Robinson--Thrall formula \cite{FRT54} and the Plancherel
measure on partitions \cite{BOO00}. 

Some of the above-mentioned special cases have nice generalisations to 
the Macdonald polynomial or the $t=q$ (i.e., Schur) level.
For example, if we replace $u$ by $-u/qT$ in \eqref{Eq_qtNO}, then 
let $T$ tend to $0$ and finally extract coefficients of $u^n$ we obtain
a $q,t$-analogue of \eqref{Eq_nfactorial}
\[
\sum_{\la\vdash n} \frac{q^{n(\la')} t^{n(\la)}}
{c_{\la}(q,t)c'_{\la}(q,t)}=[u^n] (-u;q,t)_{\infty}=
e_n\bigg(\bigg[\frac{1}{(1-q)(1-t)}\bigg]\bigg),
\]
where $e_n(x)$ is the $n$-th elementary symmetric function
and $1/(1-q)(1-t)$ is plethystic notation for the cartesian
product of the alphabets $\{1,q,q^2,\dots\}$ and 
$\{1,t,t^2,\dots\}$.
As an identity this is not actually new --- it for example 
follows by specialising $x=t^{\rho}$, $y=q^{\rho}$, $T=u$ 
in the dual Cauchy identity \eqref{Eq_dual-Cauchy} and using
the large large-$n$ limit of \eqref{Eq_MacPS}, see also
\cite{GH96} ---
but the point is that it is contained in \eqref{Eq_qtNO}.

Another interesting special case corresponds to $u=q^{-p}$ for $p$
a positive integer. Then the summand of \eqref{Eq_qtNO} contains the
factor
\[
\prod_{s\in\la}\big(1-q^{a(s)-p+1}t^{l(s)}\big),
\]
which vanishes unless $\la$ is a partition
such that $\la_i-\la_{i+1}\leq p-1$ for all $1\leq i\leq l(\la)$. 
In other words, consecutive parts should differ
by at most $p-1$ and also the smallest part has size at most $p-1$.
If we denote this set of partitions by $D_p$ (for example, $D_1=\{0\}$,
$D_2=\{\la: \la'\text{ is strict}\}$, and the number of partitions
in $D_p$ of length $l$ is $p^l-p^{l-1}$), then
\begin{equation}\label{Eq_Dp}
\sum_{\la\in D_p} T^{\abs{\la}}
\prod_{s\in\la}\frac{(1-q^{a(s)-p+1}t^{l(s)})(1-q^{a(s)+p}t^{l(s)+1})}
{(1-q^{a(s)+1}t^{l(s)})(1-q^{a(s)}t^{l(s)+1})}  
=\prod_{i=1}^{p-1} \frac{(q^{i-p}T;t,T)_{\infty}}{(q^itT;t,T)_{\infty}}.
\end{equation}
A much stronger restriction results if we take $q=t$ in \eqref{Eq_Dp}. 
Then partitions with hook-lengths equal to $p$ vanish. Partitions
with no such hook-lengths are known as $p$-cores and play an
important role in the modular representation theory of the symmetric
group, see e.g., \cite{Nakayama41,Robinson38}. Thus, with $C_p$ denoting
the set of $p$-cores,
\begin{multline}\label{Eq_Cp}
\sum_{\la\in C_p} T^{\abs{\la}}
\prod_{h\in\mathscr{H}(\la)}\frac{(1-t^{h-p})(1-t^{h+p})}{(1-t^h)^2} \\
=(T;T)_{\infty}^{p-1}
\prod_{1\leq i<j\leq p}(t^{j-i}T,t^{i-j}T;T)_{\infty}.
\end{multline}
The set of $2$-cores is given by $C_2=\{\delta_n:n\geq 1\}$, and
for $p=2$ we thus recover the Jacobi triple product identity \cite{GR04}
\[
\sum_{n\geq 1} (-1)^n T^{\binom{n}{2}}\frac{t^n-t^{1-n}}{1-t}
=(T,tT,t^{-1}T;T)_{\infty}.
\]
More generally, \eqref{Eq_Cp} is the Macdonald identity for the
affine root system $\mathrm{A}_{p-1}^{(1)}$ \cite{Macdonald72}
specialised as
\[
\eup^{-\alpha_0}\mapsto Tt^{1-p},\quad
\eup^{-\alpha_1},\dots,\eup^{-\alpha_{p-1}}\mapsto t,
\]
where $\alpha_0,\dots,\alpha_{p-1}$ are the simple roots.
This can be seen using a well-known parametrisation of $p$-cores
due to Klyachko \cite{Klyachko82} and ``Bijection 2'' from the work
of Garvan, Kim and Stanton \cite{GKS90}. For more details we also refer
to \cite{DH11,Han10,HJ11}. Identity \eqref{Eq_Dp} should thus be regarded as a
generalisation of the Jacobi triple product identity and the specialised
Macdonald identity of type $\mathrm{A}$.

\medskip

After completion of an earlier version of this paper, Amer Iqbal informed us 
of his joint work with Koz\c{c}az and Shabbir \cite{IKS10} on the refined 
topological vertex. This is defined as the rational function
\begin{multline*}
C_{\la\mu\nu}(t,q):=
q^{n(\mu')+n(\nu')+\frac{1}{2}(\abs{\la}+\abs{\mu}+\abs{\nu})} 
t^{-n(\mu)} c_{\nu}(q,t)^{-1}  \\ \times
\sum_{\eta} t^{-\abs{\eta}} s_{\la'/\eta}\big(t^{\rho} q^{-\nu}\big)
s_{\mu/\eta}\big(q^{\rho}t^{-\nu'}\big),
\end{multline*}
(where $s_{\la/\mu}$ is a skew Schur function) and reduces to the 
ordinary topological vertex \cite{AKMV05,ORV06} for $t=q$.
In their paper Iqbal \emph{et al.} use geometric considerations as
a heuristic to generate identities
for the refined topological vertex. This in turn leads
to numerous $q,t$-hook-length formulas, see \cite[Section 6]{IKS10}.
As remarked in their paper, these identities are not rigorously proved, but
checked up to some fixed order in the parameters using a computer.
Their Example 3, arising from a $5$-dimensional $U(1)$ gauge theory
is, up to a renaming of the variables, precisely our 
\eqref{Eq_qtNO}.\footnote{In the subsequent paper \cite{INRS12} Iqbal 
\emph{et al.} prove this for $t=q$ using the cyclic symmetry of the
ordinary topological vertex.}

Macdonald polynomials can also be applied to deal with the other identities
from \cite{IKS10}, and below we discuss in detail \cite[Example 4]{IKS10} 
arising from a $5$-dimensional supersymmetric $U(1)$ gauge 
theory with two hypermultiplets.
\begin{proposition}\label{Prop_IKB}
We have
\begin{subequations}\label{Eq_drie}
\begin{align}\label{Eq_drie-1}
&\sum_{\mu,\nu} \frac{(-u)^{\abs{\mu}}(-v)^{\abs{\nu}}
q^{n(\mu')+n(\nu')} t^{n(\mu)+n(\nu)}}
{c_{\mu}(q,t)c'_{\mu}(q,t)c_{\nu}(q,t)c'_{\nu}(q,t)}
\prod_{i,j\geq 1} \big(1-w q^{i-\nu_j} t^{j-\mu'_i}\big) \\
\label{Eq_drie-2}
&\quad=\sum_{\la} 
\frac{(-wqt)^{\abs{\la}} q^{n(\la')} t^{n(\la)}}
{c_{\la}(q,t)c'_{\la}(q,t)} 
\prod_{i,j\geq 1} \big(1-uq^{i-1}t^{j-\la_i'-1}\big)
\big(1-vq^{i-\la_j-1}t^{j-1}\big) \\
&\quad=\frac{(u,v,wqt,uvw;q,t)_{\infty}}{(uwq,vwt;q,t)_{\infty}}.
\label{Eq_drie-3}
\end{align}
\end{subequations}
\end{proposition}
By applying the `flop transition' to this theory, see \cite[p.~450]{IKS10},
Iqbal \emph{et al.} also obtained the following companion identity.
\begin{proposition}\label{Prop_flop}
We have
\begin{multline}\label{Eq_flop}
\sum_{\la} \frac{w^{\abs{\la}} t^{2n(\la)}}{c_{\la}(q,t)c'_{\la}(q,t)}
\prod_{i,j\geq 1} \big(1-uq^{i-1}t^{j-\la_i'}\big)
\big(1-vq^{i-1}t^{j-\la_i'}\big) \\
=\frac{(ut,vt,uw,vw;q,t)_{\infty}}{(w,uvw;q,t)_{\infty}}.
\end{multline}
\end{proposition}

For reasons that will become clear later, we first prove the second 
proposition.
\begin{proof}[Proof of Proposition~\ref{Prop_flop}]
By \eqref{Eq_qfac} the claim may also be stated as
\begin{equation}\label{Eq_same}
\sum_{\la} \frac{w^{\abs{\la}} t^{2n(\la)}(u,v;q,t)_{\la}}
{c_{\la}(q,t)c'_{\la}(q,t)}
=\prod_{i,j\geq 1}\frac{(uw,vw;q,t)_{\infty}}{(w,uvw;q,t)_{\infty}},
\end{equation}
where $(a_1,\dots,a_k;q,t)_{\la}:=(a_1;q,t)_{\la}\cdots (a_k;q,t)_{\la}$.
The shortest proof of this is to start with the 
Cauchy identity \eqref{Eq_Cauchy-noskew} and carry out the
plethystic substitutions $x\mapsto (w-uw)/(1-t)$ and $y\mapsto (1-v)/(1-t)$.
By \cite[p.~338]{Macdonald95}
\[
P_{\la}\Big(\Big[\frac{a-b}{1-t}\Big];q,t\Big)=a^{\abs{\la}}
\frac{t^{n(\la)}(b/a;q,t)_{\la}}
{c_{\la}(q,t)},
\]
and the simple relation $Q_{\la}=b_{\la}P_{\la}$ (see \eqref{Eq_QP-skew})
the identity \eqref{Eq_same} immediately follows.

It is in fact not hard to show that \eqref{Eq_same} and hence also 
\eqref{Eq_flop} admit a bounded analogue in which $\la$ is summed
over partitions of length at most $n$.
To this end we recall the symmetric rational function $R_{\la}(x;b;q,t)$ 
defined by the branching formula \cite{LW11}, 
\[
R_{\la}(x_1,\dots,x_n;b;q,t)=\sum_{\mu\subseteq\la}
\frac{(bx_n/t;q,t)_{\mu}}{(bx_n;q,t)_{\la}}\,
P_{\la/\mu}(x_n;q,t)
R_{\mu}(x_1,\dots,x_{n-1};b;q,t)
\]
and initial condition $R_{\la}(\text{--}\,;b;q,t)=\delta_{\la,0}$.
Note that $R_{\la}(x;0;q,t)=P_{\la}(x;q,t)$ and
$R_{(k)}(x_1;b;q,t)=x^k/(bx;q)_k$.
According to \cite[Corollary 5.4]{LW11}, the function $R_{\la}(x;b;q,t)$
admits the following $\mathfrak{sl}_n$ analogue of the classical $q$-Gauss sum:
\[
\sum_{\la} t^{n(\la)} \Big(\frac{c}{ab}\Big)^{\abs{\la}} 
\frac{(a,b;q,t)_{\la}}{c'_{\la}(q,t)} \, R_{\la}(x;c,q,t)
=\prod_{i=1}^n \frac{(cx_i/a,cx_i/b;q)_{\infty}}{(cx_i,cx_i/ab;q)_{\infty}}.
\]
Specialising $x=t^{\delta_n}$, using \cite[Proposition 4.4]{LW11}
\[
R_{\la}(t^{\delta_n};b;q,t)=\frac{t^{n(\la)} (t^n;q,t)_{\la}}
{(bt^{n-1};q,t)_{\la} c_{\la}(q,t)},
\]
and finally replacing $(a,b,c)\mapsto (u,v,uvw)$, yields
\[
\sum_{\la} \frac{w^{\abs{\la}} t^{2n(\la)} (t^n,u,v;q,t)_{\la}}
{(uvwt^{n-1};q,t)_{\la} \, c_{\la}(q,t)c'_{\la}(q,t)} 
=\prod_{i=1}^n \frac{(uwt^{i-1},vwt^{i-1};q)_{\infty}}
{(wt^{i-1},uvwt^{i-1};q)_{\infty}},
\]
where we note that $(t^n;q,t)_{\la}=0$ unless $l(\la)\leq n$.
In the large-$n$ limit this gives \eqref{Eq_same}.
\end{proof}

\begin{proof}[Proof of Proposition~\ref{Prop_IKB}]
In the following we denote the double sum in \eqref{Eq_drie-1} by
\textrm{LHS}. Replacing $\mu\mapsto\mu'$ and using 
\eqref{Eq_cdual} as well as 
\begin{equation}\label{Eq_PSlim}
P_{\la}(t^{\rho};q,t)=\frac{t^{n(\la)}}{c_{\la}(q,t)}
\end{equation}
(this is the large-$n$ limit of \eqref{Eq_MacPS}), we get
\[
\textrm{LHS}=
\sum_{\mu,\nu} \frac{(-u)^{\abs{\mu}}(-v)^{\abs{\nu}}
q^{n(\nu')} t^{n(\mu')} P_{\mu}(q^{\rho};t,q)P_{\nu}(t^{\rho};q,t)}
{c'_{\mu}(t,q)c'_{\nu}(q,t)}
\prod_{i,j\geq 1} \big(1-w q^{i-\nu_j} t^{j-\mu_i}\big).
\]
In order to decouple the sums over $\mu$ and $\nu$ we apply the
dual Cauchy identity \eqref{Eq_dual-Cauchy}
with $(x,y,T)\mapsto (q^{-\nu}t^{\rho},q^{\rho}t^{-\mu},-wqt)$.
Then
\begin{multline*}
\textrm{LHS}=
\sum_{\la,\mu,\nu} (-wqt)^{\abs{\la}}(-u)^{\abs{\mu}}(-v)^{\abs{\nu}}
q^{n(\nu')} t^{n(\mu')} \\
\times \frac{P_{\mu}(q^{\rho};t,q)P_{\la'}(q^{\rho}t^{-\mu};t,q)
P_{\nu}(t^{\rho};q,t)P_{\la}(q^{-\nu}t^{\rho};q,t)}
{c'_{\mu}(t,q)c'_{\nu}(q,t)}.
\end{multline*}
By a double application of the Macdonald--Koornwinder duality \eqref{Eq_Tom}
(with $\rho_n\mapsto \rho$) this can be transformed into
\begin{multline*}
\textrm{LHS}=
\sum_{\la,\mu,\nu} (-wqt)^{\abs{\la}}(-u)^{\abs{\mu}}(-v)^{\abs{\nu}}
q^{n(\nu')} t^{n(\mu')} \\
\times \frac{P_{\la'}(q^{\rho};t,q)P_{\mu}(q^{\rho}t^{-\la'};t,q)
P_{\la}(t^{\rho};q,t)P_{\nu}(q^{-\la}t^{\rho};q,t)}
{c'_{\mu}(t,q)c'_{\nu}(q,t)}.
\end{multline*}
Specialising $T=-1$ and $y=q^{\rho}$ in \eqref{Eq_dual-Cauchy}
and using \eqref{Eq_PSlim}, we obtain the following
$q,t$-analogue of Euler's $q$-exponential sum 
(see also \cite[p.~294]{Lassalle98}):
\[
\sum_{\la} \frac{(-1)^{\la} q^{n(\la')} P_{\la}(x;q,t)}{c'_{\la}(q,t)}
=\prod_{i\geq 1} (x_i;q)_{\infty}.
\]
This can be used to carry out the sums over $\mu$ and $\nu$, 
resulting in
\begin{align*}
\textrm{LHS}&=\sum_{\la} (-wqt)^{\abs{\la}}
P_{\la}(t^{\rho};q,t) P_{\la'}(q^{\rho};t,q) 
\prod_{i\geq 1} (uq^{i-1}t^{-\la'_i};t)_{\infty} 
(vq^{-\la_i}t^{i-1};q)_{\infty} \\
&=\sum_{\la} \frac{(-wqt)^{\abs{\la}} q^{n(\la')} t^{n(\la)}}
{c_{\la}(q,t) c'_{\la}(q,t)} 
\prod_{i,j\geq 1} \big(1-uq^{i-1}t^{j-\la'_i-1}\big)
\big(1-vq^{i-\la_j-1}t^{j-1}\big),
\end{align*}
where in the second step we have once again used \eqref{Eq_PSlim} 
followed by \eqref{Eq_cdual}. 
This proves the equality between \eqref{Eq_drie-1} and \eqref{Eq_drie-2}.
In fact, the entire proof is now done since the equality of
\eqref{Eq_drie-2} and \eqref{Eq_drie-3} is equivalent to the identity
\eqref{Eq_flop} arising from the flop transition.
Indeed, by \eqref{Eq_qfac} the second half of Proposition~\ref{Prop_IKB}
can also be stated as
\[
\sum_{\la} 
\frac{(-wqt)^{\abs{\la}} q^{n(\la')} t^{n(\la)}(u/t;q,t)_{\la}(v/q;t,q)_{\la'}}
{c_{\la}(q,t)c'_{\la}(q,t)} 
=\frac{(wqt,uvw;q,t)_{\infty}}{(uwq,vwt;q,t)_{\infty}}.
\]
Since
\[
(z;t,q)_{\la'}=(-z)^{\abs{\la}} q^{-n(\la')} t^{n(\la)} (z^{-1};q,t)_{\la}
\]
this is \eqref{Eq_same} in which $(u,v,w)$ has been replaced by
$(u/t,q/v,vwt)$.
\end{proof}

\appendix

\section{}

Jim Bryan suggested an alternative derivation of \eqref{Eq_qtNO} based on the
equivariant DMVV formula for the Hilbert scheme of $n$ points in the plane, 
$(\Complex^2)^{[n]}$.
This formula was first conjectured by Li, Liu and Zhou in \cite{LLZ06} and 
subsequently proved by Waelder \cite{Waelder08} as a consequence of the
equivariant MacKay correspondence.

Let $(u_1,u_2)$ be the equivariant parameters of the natural
torus action on $(\Complex^2)^{[n]}$,
and set $t_1:=\eup^{2\pi\iup u_1}$ and $t_2:=\eup^{2\pi\iup u_2}$.
Let $\Ell\big((\Complex^2)^{[n]};u,p,t_1,t_2\big)$ be the equivariant 
elliptic genus of $(\Complex^2)^{[n]}$, where
$p:=\exp(2\pi\iup\tau)$ and $u:=\exp(2\pi\iup z)$ for $\tau\in\mathbb{H}$ and
$z\in\Complex$. Treating $u,p,t_1$ and $t_2$ as formal variables, the 
equivariant DMVV formula \cite[Theorem 12]{Waelder08} expresses the generating
function for the elliptic genera as a product: 
\begin{multline}\label{Eq_Waelder}
\sum_{n\geq 0} T^n \Ell\big((\Complex^2)^{[n]};u,p,t_1,t_2\big) \\
=\prod_{m\geq 0}\prod_{k\geq 1} \prod_{\ell,n_1,n_2\in\Z}
\frac{1}{(1-p^m T^k u^{\ell} t_1^{n_1} t_2^{n_2})^{c(km,\ell,n_1,n_2)}}.
\end{multline}
The integers $c(m,\ell,n_1,n_2)$ on the right are determined by the 
equivariant elliptic genus of $\Complex^2$, given by a simple ratio of 
Jacobi theta functions:
\begin{align}\label{Eq_c-def}
\Ell(\Complex^2,u,p,t_1,t_2)&=
\frac{\theta(ut_1^{-1},u^{-1}t_2;p)}{\theta(t_1^{-1},t_2;p)} \\
&=\sum_{m\geq 0}\, \sum_{\ell,n_1,n_2\in\Z} c(m,\ell,n_1,n_2)
p^m u^{\ell} t_1^{n_1} t_2^{n_2}, \notag
\end{align}
where
\[
\theta(u;p):=\sum_{k\in\Z} (-u)^k p^{\binom{k}{2}}=(u,p/u,p;p)_{\infty}
\]
and
\[
\theta(u_1,\dots,u_k;p):=\theta(u_1;p)\cdots\theta(u_k;p).
\]

In \cite{LLZ06} an explicit formula in terms of arm and leg-lengths
is obtained for the generating function (over $n$)
of elliptic genera of the framed moduli spaces $M(r,n)$ 
of torsion-free sheaves on $\mathbb{P}^2$ of rank $r$ and second Chern 
class $n$, see \cite{NY05}. 
Since $M(1,n)$ coincides with $(\Complex^2)^{[n]}$ this implies
\cite[Equation (2.4); $\mu\mapsto\la'$]{LLZ06}
\begin{multline}\label{Eq_theta-arms-legs}
\sum_{n\geq 0} T^n \Ell\big((\Complex^2)^{[n]};u,p,t_1,t_2\big) \\
=\sum_{\la} T^{\abs{\la}} \prod_{s\in\la}
\frac{\theta(ut_1^{-a(s)-1}t_2^{l(s)},u^{-1}t_1^{-a(s)} t_2^{l(s)+1};p)}
{\theta(t_1^{-a(s)-1}t_2^{l(s)},t_1^{-a(s)} t_2^{l(s)+1};p)}.
\end{multline}

Combining \eqref{Eq_Waelder} with \eqref{Eq_theta-arms-legs} we can derive 
an elliptic analogue of the Nekrasov--Okounkov formula as follows.
Define a second set of integers $C(m,\ell,n_1,n_2)$ by
\begin{multline}\label{Eq_C-def}
\frac{(put_1^{-1},pu^{-1}t_1,put_2^{-1},pu^{-1}t_2;p)_{\infty}}
{(pt_1^{-1},pt_1,pt_2^{-1},pt_2;p)_{\infty}} \\
=\sum_{m\geq 0} \, \sum_{\ell,n_1,n_2\in\Z} C(m,\ell,n_1,n_2)
p^m u^{\ell} t_1^{n_1} t_2^{n_2}.
\end{multline}
From the invariance of the left-hand side under the substitutions
$(u,t_1,t_2)\mapsto (u,t_2,t_1)$ and
$(u,t_1,t_2)\mapsto (u^{-1},t_1^{-1},t_2^{-1})$
it follows that
\[
C(m,\ell,n_1,n_2)=C(m,\ell,n_2,n_1)=C(m,-\ell,-n_1,-n_2).
\]
By \eqref{Eq_c-def} and $\theta(u;p)=(1-u)(pu,pu^{-1};p)_{\infty}$,
\begin{align*}
\Ell&(\Complex^2,u,p,t_1,t_2) \\
&=\frac{(1-ut_1^{-1})(1-u^{-1}t_2)}{(1-t_1)(1-t_2)}\cdot
\frac{(put_1^{-1},pu^{-1}t_1,put_2^{-1},pu^{-1}t_2;p)_{\infty}}
{(pt_1^{-1},pt_1,pt_2^{-1},pt_2;p)_{\infty}} \\
&=\frac{(1-ut_1^{-1})(1-u^{-1}t_2)}{(1-t_1^{-1})(1-t_2)} 
\sum_{m\geq 0} \, \sum_{\ell,n_1,n_2\in\Z} C(m,\ell,n_1,n_2)
p^m u^{\ell} t_1^{n_1} t_2^{n_2} \\
&=
\sum_{m\geq 0} \, \sum_{\ell,n_1,n_2\in\Z} \,
\sum_{i,j\geq 1} D(m,\ell,n_1+i,n_2-j) p^m u^{\ell} t_1^{n_1} t_2^{n_2},
\end{align*}
where
\begin{multline*}
D(m,\ell,n_1,n_2):=
C(m,\ell,n_1-1,n_2+1)+C(m,\ell,n_1,n_2) \\
-C(m,\ell-1,n_1,n_2+1)-C(m,\ell+1,n_1-1,n_2).
\end{multline*}
Comparison with \eqref{Eq_Waelder} yields
\[
c(m,\ell,n_1,n_2)=\sum_{i,j\geq 1} D(m,\ell,n_1+i,n_2-j).
\]
Hence
\begin{align*}
&\prod_{m\geq 0} \, \prod_{k\geq 1} \, \prod_{\ell,n_1,n_2\in\Z}
\frac{1}{(1-p^m T^k u^{\ell} t_1^{n_1} t_2^{n_2})^{c(km,\ell,n_1,n_2)}} \\
&=\prod_{m\geq 0}\, \prod_{i,j,k\geq 1} \, \prod_{\ell,n_1,n_2\in\Z}
\frac{1}{(1-p^m T^k u^{\ell} t_1^{n_1} t_2^{n_2})^
{D(km,\ell,n_1+i,n_2-j)}} \\
&=\prod_{m\geq 0}\,\prod_{i,j,k\geq 1} \, \prod_{\ell,n_1,n_2\in\Z} 
\bigg(
\frac{(1-p^m T^k u^{\ell+1} t_1^{n_1-i} t_2^{n_2+j-1})}
{(1-p^m T^k u^{\ell} t_1^{n_1-i+1} t_2^{n_2+j-1})} \\
&\qquad\qquad\qquad\qquad\qquad\quad
\times\frac{(1-p^m T^k u^{\ell-1} t_1^{n_1-i+1} t_2^{n_2+j})}
{(1-p^m T^k u^{\ell} t_1^{n_1-i} t_2^{n_2+j})}\bigg)^{C(km,\ell,n_1,n_2)}. 
\end{align*}
Equating the right-hand sides of \eqref{Eq_Waelder} and 
\eqref{Eq_theta-arms-legs}, using the above rewriting of the former,
and finally replacing $(t_1,t_2)\mapsto(q^{-1},t)$ yields
\begin{align*}
\sum_{\la} & T^{\abs{\la}} \prod_{s\in\la}
\frac{\theta(uq^{a(s)+1}t^{l(s)},u^{-1}q^{a(s)} t^{l(s)+1};p)}
{\theta(q^{a(s)+1}t^{l(s)},q^{a(s)} t^{l(s)+1};p)} \\
&=\prod_{m\geq 0}\prod_{i,j,k\geq 1} \, \prod_{\ell,n_1,n_2\in\Z} 
\bigg(\frac{(1-p^m T^k u^{\ell+1} q^{i-n_1} t^{j+n_2-1})}
{(1-p^m T^k u^{\ell} q^{i-n_1-1} t^{j+n_2-1})} \\
&\qquad\qquad\qquad\qquad\qquad\quad
\times\frac{(1-p^m T^k u^{\ell-1} q^{i-n_1-1} t^{j+n_2})}
{(1-p^m T^k u^{\ell} q^{i-n_1} t^{j+n_2})}\bigg)^{C(km,\ell,n_1,n_2)}.
\end{align*}
Since the left-hand side of \eqref{Eq_C-def} trivialises to $1$ when
the elliptic nome $p$ tends to $0$,
\[
C(0,\ell,n_1,n_2)=\delta_{\ell,0}\delta_{n_1,0}\delta_{n_2,0}.
\]
In the $p\to 0$ limit the above result thus simplifies to \eqref{Eq_qtNO}.


\begin{thebibliography}{99}

\bibitem{AKMV05}
M. Aganagic, A. Klemm, M. Mari\~{n}o, C. Vafa,
\textit{The topological vertex},
Comm. Math. Phys. \textbf{254} (2005), 425--478. 

\bibitem{BF99}
T. H. Baker and P. J. Forrester, 
\textit{Transformation formulas for multivariable basic hypergeometric series}, 
Methods Appl. Anal. \textbf{6} (1999), 147--164.

\bibitem{BOO00}
A. Borodin, A. Okounkov and G. Olshanski,
\textit{Asymptotics of Plancherel measures for symmetric groups},
J. Amer. Math. Soc. \textbf{13} (2000), 481--515.

\bibitem{BKW16}
R. P. Brent, C. Krattenthaler  and O. Warnaar,
\textit{Discrete analogues of Macdonald--Mehta integrals},
J. Combin. Theory Ser. A \textbf{144} (2016), 80--138.

\bibitem{CNO14}
E. Carlsson, N. Nekrasov and A. Okounkov,
\textit{Five dimensional gauge theories and vertex operators},
Mosc. Math. J. \textbf{14} (2014), 39--61.

\bibitem{CRV16}
E. Carlsson and F. Rodriguez-Villegas,
\textit{Vertex operators and character varieties},
preprint.

\bibitem{Cadogan71}
C. C. Cadogan, 
\textit{The M\"{o}bius function and connected graphs},
J. Combinatorial Theory Ser. B \textbf{11} (1971), 193--200.

\bibitem{CDDP15}
W. Chuang, D. Diaconescu, R. Donagi and T. Pantev,
\textit{Parabolic refined invariants and Macdonald polynomials}, 
Commun. Math. Phys. \textbf{335} (2015) 1323--1379.

\bibitem{DH11}
P.-O. Dehaye and G.-N. Han,
\textit{A multiset hook length formula and some applications},
Discrete Math. \textbf{311} (2011), 2690--2702.

\bibitem{FRT54}
J. S. Frame, G. de B. Robinson and R. M. Thrall,
\textit{The hook graphs of the symmetric groups},
Canadian J. Math. \textbf{6} (1954), 316--324.

\bibitem{GH96}
A. M. Garsia and M. Haiman,
\textit{A remarkable $q,t$-Catalan sequence and $q$-Lagrange inversion},
J. Algebraic Combin. \textbf{5} (1996), 191--244.

\bibitem{GHT99}
A. M. Garsia, M. Haiman, and G. Tesler,
\textit{Explicit plethystic formulas for Macdonald $q,t$-Kostka coefficients},
S\'em. Lothar. Combin. \textbf{42} (1999), Art.~B42m, 45 pp. 

\bibitem{GKS90}
F. Garvan, D. Kim and D. Stanton, 
\textit{Cranks and $t$-cores}, 
Invent. Math. \textbf{101} (1990) 1--17.

\bibitem{GR04}
G. Gasper and M. Rahman,
\textit{Basic Hypergeometric Series}, second edition,
Encyclopedia of Mathematics and its Applications, Vol.~96,
Cambridge University Press, Cambridge, 2004.

\bibitem{Getzler95}
E. Getzler,
\textit{Mixed Hodge structures of configuration spaces},
\href{http://arxiv.org/abs/alg-geom/9510018}{arXiv:9510018}.

\bibitem{Gothen94}
P. B. Gothen, 
\textit{The Betti numbers of the moduli space of rank 3 Higgs bundles},
Internat. J. Math. \textbf{5} (1994), 861--875.

\bibitem{Haglund08} 
J. Haglund,
\textit{The $q,t$-Catalan Numbers and the Space of Diagonal Harmonics},
University lecture series, Vol.~41,
AMS, Providence, RI, 2008.

\bibitem{Han10}
G.-N. Han,
\textit{The Nekrasov--Okounkov hook length formula: 
refinement, elementary proof, extension and applications},
Ann. Inst. Fourier (Grenoble) \textbf{60} (2010), 1--29. 

\bibitem{HJ11}
G.-N. Han and K. Q. Ji,
\textit{Combining hook length formulas and BG-ranks for partitions via 
the Littlewood decomposition},
Trans. Amer. Math. Soc. \textbf{363} (2011), 1041--1060. 

\bibitem{HLRV11}
T. Hausel, E. Letellier and F. Rodriguez-Villegas,
\textit{Arithmetic harmonic analysis on character and quiver varieties},
Duke Math. J. \textbf{160} (2011), 323--400.

\bibitem{HLRV13}
T. Hausel, E. Letellier and F. Rodriguez-Villegas,
\textit{Arithmetic harmonic analysis on character and quiver varieties II},
Adv. Math. \textbf{234} (2013), 85--128.

\bibitem{HRV08}
T. Hausel and F. Rodriguez-Villegas,
\textit{Mixed Hodge polynomials of character varieties},
(with an appendix by N.~M.~Katz.),
Invent. Math. \textbf{174} (2008), 555--624. 

\bibitem{Hitchin87}
N. Hitchin,
\textit{The self-duality equations on a Riemann surface},
Proc. London Math. Soc. (3) \textbf{55} (1987) 59--126.

\bibitem{IKS10}
A. Iqbal, C. Koz\c{c}az and K. Shabbir, 
\textit{Refined topological vertex, cylindric partitions and $U(1)$ 
adjoint theory},
Nuclear Phys. B \textbf{838} (2010), 422--457.

\bibitem{INRS12}
A. Iqbal, S. Nazir, Z. Raza, Z. Saleem,
\textit{Generalizations of Nekrasov--Okounkov identity},
Ann. Comb. \textbf{16} (2012), 745--753. 

\bibitem{Kaneko96}
J. Kaneko,
\textit{$q$-Selberg integrals and Macdonald polynomials},
Ann. Sci. \'Ecole Norm. Sup. (4) \textbf{29} (1996), 583--637. 

\bibitem{Klyachko82}
A. A. Klyachko,
\textit{Modular forms and representations of symmetric groups},
J. Soviet Math. \textbf{26} (1984), 1879--1887.

\bibitem{Knop97}
F. Knop, 
\textit{Symmetric and non-symmetric quantum Capelli polynomials},
Comment. Math. Helv. \textbf{72} (1997), 84--100.

\bibitem{KS96}
F. Knop and S. Sahi,
\textit{Difference equations and symmetric polynomials defined by their zeros}, 
Internat. Math. Res. Notices 1996, 473--486. 

\bibitem{Knuth70}
D. E. Knuth, 
\textit{Permutations, matrices, and generalized Young tableaux},
Pacific J. Math. \textbf{34} (1970), 709--727.

\bibitem{Lascoux01} 
A. Lascoux,
\textit{Symmetric Functions and Combinatorial Operators on Polynomials}, 
CBMS Regional Conference Series in Mathematics Vol.~99, 
AMS, Providence, RI, 2003.

\bibitem{LW11}
A. Lascoux and S. O. Warnaar,
\textit{Branching rules for symmetric Macdonald polynomials and $\mathfrak{sl}_n$
basic hypergeometric series}, 
Adv. in Applied Math. \textbf{46} (2011), 424--456.

\bibitem{Lassalle98}
M. Lassalle,
\textit{Coefficients binomiaux g\'{e}n\'{e}ralis\'{e}s et polyn\^{o}mes
de Macdonald},
J. Funct. Anal. \textbf{158} (1998), 289--324.

\bibitem{LLZ06}
J. Li, K. Liu and J. Zhou,
\textit{Topological string partition functions as equivariant indices},
Asian J.  Math. \textbf{10} (2006), 81--114.

\bibitem{Littlewood50}
D. E. Littlewood, 
\textit{The Theory of Group Characters and Matrix Representations of
Groups}, Oxford University Press, 1950.

\bibitem{Macdonald72}
I. G. Macdonald,
\textit{Affine root systems and Dedekind's $\eta$-function},
Invent. Math. \textbf{15}, 91--143, (1972).

\bibitem{Macdonald95}
I. G. Macdonald,
\textit{Symmetric functions and Hall polynomials},
2nd edition, Oxford Univ. Press, New York, 1995.

\bibitem{Macdonald13}
I. G. Macdonald,
\textit{Hypergeometric Functions II ($q$-analogues)},
\href{http://arxiv.org/abs/1309.5208}{arXiv:1309.5208}.

\bibitem{NY05}
H. Nakajima and K. Hiraku,
\textit{Instanton counting on blowup. I. $4$-dimensional pure gauge theory},
Invent. Math. \textbf{162} (2005), 313--355. 

\bibitem{Nakayama41}
T. Nakayama
\textit{On some modular properties of irreducible representations of a
symmetric group. I \& II}, 
Jap. J. Math. \textbf{18} (1941), 89--108 \& \textbf{17} (1941), 411--423.

\bibitem{NO06}
N. A. Nekrasov and A. Okounkov,
\textit{Seiberg--Witten theory and random partitions},
in \textit{The Unity of Mathematics}, pp.~525--596, 
Progr. Math., Vol.~ 244, Birkh\"auser Boston, Boston, MA, 2006. 

\bibitem{Okounkov97}
A. Okounkov,
\textit{Binomial formula for Macdonald polynomials and applications},
Math. Res. Lett. \textbf{4} (1997), 533--553.

\bibitem{Okounkov98}
A. Okounkov,
\textit{(Shifted) Macdonald polynomials: 
$q$-integral representation and combinatorial formula},
Compositio Math. \textbf{112} (1998), 147--182. 

\bibitem{ORV06}
A. Okounkov, N. Reshetikhin and C. Vafa,
\textit{Quantum Calabi--Yau and classical crystals},
in \textit{The unity of mathematics}, pp.~597--618, 
Progr. Math., Vol.~244, Birkh\"user Boston, Boston, MA, 2006. 

\bibitem{Rains05}
E. M. Rains,
\textit{BC$_n$-symmetric polynomials},
Transform. Groups \textbf{10} (2005), 63--132.

\bibitem{Robinson38}
G. de B. Robinson,
\textit{On representations of the symmetric group}, 
Amer. J. Math. \textbf{60}, (1938), 745--760.

\bibitem{Sahi96}
S. Sahi,
\textit{Interpolation, integrality, and a generalisation of Macdonald's 
polynomials},
Internat. Math. Res. Notices \textbf{1996}, 457--471.

\bibitem{Schensted61}
C. Schensted,
\textit{Longest increasing and decreasing subsequences},
Canad. J. Math. \textbf{13} (1961), 179--191.

\bibitem{Stanley71}
R. P. Stanley, 
\textit{Theory and applications of plane partitions: Part 2}, 
Stud. Appl. Math \textbf{50} (1971), 259--279.

\bibitem{Stanley99}
R. P. Stanley, 
\textit{Enumerative Combinatorics. Vol.~2},
Cambridge Studies in Advanced Mathematics, Vol.~62,
Cambridge University Press, Cambridge, 1999.

\bibitem{Waelder08}
R. Waelder,
\textit{Equivariant elliptic genera and local McKay correspondences},
Asian J. Math. \textbf{12} (2008), 251--284. 

\bibitem{Westbury06}
B. W. Westbury,
\textit{Universal characters from the Macdonald identities},
Adv. Math. \textbf{202} (2006), 50--63.

\end{thebibliography}
\end{document}